\documentclass[a4, 11pt]{article}
\usepackage [hmargin=2cm,vmargin=3cm]{geometry}
\usepackage{lmodern}		
\usepackage[T1]{fontenc}		
\usepackage{amsmath}
\usepackage{comment}
\usepackage[colorlinks = true,
linkcolor = blue,
urlcolor  = blue,
citecolor = blue,
anchorcolor = blue]{hyperref}
\usepackage{amsthm}
\PassOptionsToPackage{hyphens}{url}\usepackage{hyperref}
\usepackage{caption}
\usepackage{mathrsfs}
\usepackage{amssymb}
\usepackage{amsthm}
\usepackage{enumerate}
\usepackage{mathtools}
\usepackage{amssymb}
\usepackage{xcolor}
\usepackage{bbm}
\usepackage{xstring}
\allowdisplaybreaks

\numberwithin{equation}{section}		
\newtheorem{theorem}{Theorem}[section]

\newtheorem{prop}[theorem]{Proposition}
\newtheorem{corollary}[theorem]{Corollary}
\newtheorem{lemma}[theorem]{Lemma}

\theoremstyle{definition}				
\newtheorem{remark}[theorem]{Remark}

\newtheorem{definition}[theorem]{Definition}
\newtheorem*{definition*}{Definition}

\theoremstyle{theorem}
\newcounter{mtheorem}
\newtheorem{mtheorem}[mtheorem]{Theorem}

\newcommand{\Z}{\ensuremath{\mathbb{Z}}}

\newcommand{\R}{\ensuremath{\mathbb{R}}}
\newcommand{\W}{\ensuremath{\mathcal{W}}}

\newcommand{\OO}{\ensuremath{\mathcal{O}}}

\newcommand{\C}{\ensuremath{\mathbb{C}}}
\newcommand{\CP}{\ensuremath{\mathbb{CP}}}

\newcommand{\m}{\ensuremath{\mathfrak{m}}}

\renewcommand{\k}{\mathbf{k}}

\newcommand{\morm}[1]{{| #1 |}}

\newcommand{\gen}[1]{{\left\langle #1 \right\rangle}}

\renewcommand{\leq}{\leqslant }
\renewcommand{\geq}{\geqslant }
\renewcommand{\phi}{\varphi}
\renewcommand{\rho}{\varrho}

\newcommand{\sff}{\mathrm{I\!I}}
\let\H\relax
\DeclareMathOperator{\H}{H}
\let\Re\relax
\DeclareMathOperator{\Re}{Re}
\let\Im\relax
\DeclareMathOperator{\Im}{Im}

\DeclareMathOperator{\Span}{span}

\DeclareMathOperator{\diam}{diam}

\DeclareMathOperator{\tr}{Tr}

\DeclareMathOperator{\Rm}{Rm}

\DeclareMathOperator{\Ric}{Ric}

\DeclareMathOperator{\vol}{Vol}
\DeclareMathOperator{\dvol}{dvol}
\DeclareMathOperator{\dA}{dA}
\DeclareMathOperator{\area}{Area}
\DeclareMathOperator{\length}{Length}
\DeclareMathOperator{\diag}{diag}

\DeclareMathOperator{\eucl}{Eucl}
\DeclareMathOperator{\Div}{div}

\DeclareMathOperator{\loc}{loc}
\newcommand{\del}{\partial}
\newcommand{\delbar}{\overline{\partial}}

\usepackage[
backend=biber,
maxnames=99,
style=numeric,
doi = false,
isbn = false
]{biblatex}

\DeclareFieldFormat{mrnumber}{\textbf{MR\space}#1}

\newcommand{\spliturl}[1]{%
	\StrBefore{#1}{,}[\urlfirstpart]%
	\StrBehind{#1}{,}[\urlsecondpart]%
	MR\urlfirstpart, Zbl \urlsecondpart%
}

\DeclareFieldFormat{url}{\spliturl{#1}}

\usepackage{xurl}
\hypersetup{breaklinks=true}

\title{Mass Inequality and Stability of the Positive Mass Theorem For Kähler Manifolds}
\author{Johan Jacoby Klemmensen}

\begin{document}

	\maketitle
	
	\begin{abstract}
		We prove an integral inequality and two stability results for the ADM mass on AE Kähler manifolds of all complex dimensions. The inequality bounds the ADM mass from below by an integral of the scalar curvature and the Hessian of certain holomorphic coordinate functions arising from the complex coordinates at infinity. Using this, we first prove a stability result for any sequence of AE Kähler manifolds with ADM mass converging to zero. We conclude that, for any such sequence, there exist subsets of each with vanishing boundaries in the limit such that the complements converge to Euclidean space in the pointed Gromov-Hausdorff sense. This gives the first stability result of the Positive Mass Theorem for Kähler manifolds, or more generally, of manifolds without strong curvature or volume conditions or for a very explicit family of metrics in real dimensions greater than three. If we furthermore impose a uniform lower bound on the Ricci curvature, the second stability theorem shows the same result without taking the complement of a sequence of vanishing sets. Finally, we find three new families of AE Kähler manifolds with vanishing mass in the limit for which the stability results apply.
	\end{abstract}

	\tableofcontents
	
	\section{Introduction}
	
	\subsection{Background}
	
	Let $(M^n, g)$ be a complete Riemannian manifold of dimension $n \geq 3$. We say that $(M, g)$ is \textit{asymptotically Euclidean} (AE) if the following is true: There exists a compact set $K \subset M$ such that $M \setminus K = \bigsqcup_{k=1}^{k_0} M_i$, where each \textit{end} $M_i \cong \mathbb{R}^n \setminus \overline{B_1(0)}$. Furthermore, on each end, there exists a coordinate chart $(x_1, \dots, x_n)$ such that
	\begin{equation}\label{AERiemannian}
		|\del^k(g_{ij}-\delta_{ij})| = \OO(|x|^{1-\frac{n}{2}-\delta-k}),\; \quad \delta>0, \; k=0,1,2.
	\end{equation}
	If the scalar curvature $R_g$ is integrable, Bartnik \cite{bartnikmass1986} and Chru\'sciel \cite{chruscielBoundary1986} showed that the \textit{ADM mass} $\m(g)$ for each end, defined as
	\begin{equation}\label{ADMMass}
		\m(g) \coloneqq \lim_{\rho\to\infty}\frac{\Gamma(\frac{n}{2})}{4(n-1)\pi^{n/2}} \int_{S_\rho} \left(g_{kl,k}-g_{kk,l} \right)\nu_l \dA_{\eucl},
	\end{equation}
	is well-defined and independent of the chosen asymptotically Euclidean coordinate chart at infinity. Here, $\nu$ is the outer Euclidean normal of the Euclidean sphere $S_\rho$ of radius $\rho$ in the asymptotic coordinates.\\
	
	The conjectured Positive Mass Theorem was first proven by Schoen and Yau \cite{schoenproof1979, schoenProof1981} and Witten \cite{wittennew1981} for $\dim M \leq 7$ and $M$ spin, respectively. It states the following: if the scalar curvature is nonnegative and integrable, then $\mathfrak{m}(g) \geq 0$ for every end, and $\mathfrak{m}(g) = 0$ for some end if and only if $(M, g)$ is isometric to $(\mathbb{R}^n, g)$. Hein-LeBrun later proved the theorem for AE \textit{Kähler} manifolds \cite{heinMass2016}, which is important for this work as we will study the \textit{stability} of the Positive Mass Theorem for AE Kähler manifolds.\\
	
	The question of stability is as follows: if the ADM mass is small, is $(M, g)$ close to $(\mathbb{R}^n, g)$ in some sense? The answer was complicated by the discovery of a sequence of AE rotationally symmetric three-dimensional manifolds $(M_i^3, g_i)$ such that $\mathfrak{m}(M^3_i, g_i) \to 0$, but $(M^3_i, g_i)$ develops infinitely deep gravitational wells where the GH-distance diverges \cite{leeStability2014}. It then became clear that other weaker norms need to be considered, of which Sormani-Wenger's \textit{intrinsic flat distance} is probably the most utilized. In this direction, we mention the papers \cite{leeStability2014, leflochnonlinear2015, brydenStability2021, allenIntrinsic2021, huangIntrinsic2022, pachecoIntrinsic2023}.\\

	Gromov \cite{gromovDirac2014} and Sormani \cite{sormaniConjectures2023} formulated an analogous question for the Geroch conjecture on tori, asking if a sequence of tori with a uniform lower bound on the scalar curvature going to zero converges to flat tori in the intrinsic flat distance. This turned out not to be true, as shown in the paper by Lee, Naber, and Neumayer \cite{leedp2023} on the $d_p$-distance: they built sequences of smooth tori $(\mathbb{T}^n, g_i)$ for $n \geq 4$ with scalar curvatures $R_{g_i} \geq -\frac{1}{i}$ such that $(\mathbb{T}^n, g_i)$ converges to the flat metric on $\mathbb{T}^n$ in the $d_p$-distance. In contrast, the intrinsic flat and Gromov-Hausdorff (GH) limits degenerate to a lower-dimensional torus or, depending on the sequence, even a point. Later work by Kazaras and Xu \cite{kazarasDrawstrings2023} provided a family of tori in dimension $n = 3$ with curvature bound as for Lee, Naber, and Neumayer, but with intrinsic flat distance degenerating to a flat torus with a circle pulled to a point. These examples indicate that the intrinsic flat distance may be too strong for several applications.\\
	
	Another direction in the investigation of stability came after Bray, Kazaras, Khuri, and Stern \cite{brayHarmonic2022} proved their integral inequality, inspired by the inequality of Stern \cite{sternScalar2022}. This inequality bounds the ADM mass from below by an integral involving the scalar curvature and the Hessian of certain harmonic functions. Using this, Dong \cite{dongstability2023a}, and later Dong and Song \cite{dongStability2023}, proved that any sequence of AE Riemannian three-manifolds with ADM mass going to zero converges in the Gromov-Hausdorff topology to Euclidean three-space after cutting out sets with vanishing boundary in the limit. Similar results were also proven using this integral inequality, although under other assumptions, including a curvature bound, by Allen, Bryden, and Kazaras \cite{allenStability2022} and Kazaras, Khuri, and Lee \cite{kazarasStability2021} on three-manifolds. For earlier results establishing $L^2$-estimates of the curvature outside of a compact set, we mention \cite{brayCurvature2002, finsterCurvature2002, finsterLevel2009, leenearequality2009}.
	
	\subsection{Structure of the Paper}
	
In this work, we prove two stability results for the Positive Mass Theorem and a new integral inequality for the ADM mass on AE Kähler manifolds $(X^{2m}, g, J)$. These are the first stability results of the Positive Mass Theorem for Kähler manifolds. More generally, Theorem \ref{MCorollaryStabNoRic} is the first stability theorem in higher dimensions without strong assumptions on the form of the metric (e.g., \cite{huangIntrinsic2017, brydenStability2021, brydenStability2020}) or curvature/volume conditions (e.g., \cite{allenVolume2024, allenIntrinsic2021}) in dimensions greater than three.

The stability results follow from the following integral inequality:

\begin{mtheorem}[Theorem \ref{TheoremMassBound}]\label{MTheoremMassBound}
	Let $(X^{2m},g,J)$ be an asymptotically Euclidean Kähler manifold with nonnegative and integrable scalar curvature $R_g$. Let $z = x_1 + ix_2$ be one of the holomorphic coordinate functions from Theorem \ref{BiholomorphismAsympt}. Then
	\begin{equation}\label{mMassFormula}
		\m(g) \geq \frac{(m-1)!}{4(2m-1)\pi^m} \int_X \frac{1}{2}|\nabla^2 x_1|^2 + \frac{1}{2}|\nabla^2 x_2|^2 + \morm{\nabla x_1}^2 R_g \dvol_g.
	\end{equation}
\end{mtheorem}

The formula resembles the result \cite[Theorem 1.2]{brayHarmonic2022} on three-dimensional AE Riemannian manifolds. However, \eqref{mMassFormula} is an integral over the entire manifold, and this global control allows for stronger stability statements than in the three-dimensional Riemannian case (see Theorem \ref{MCorollaryStabRicciBound} and the discussion afterward).\\

From the integral inequality \eqref{mMassFormula}, the rigidity of the Positive Mass Theorem for Kähler manifolds follows easily (Corollary \ref{CorollaryPMT}). The proof of Theorem \ref{MTheoremMassBound} includes conceptual ideas from three different proofs of the Positive Mass Theorem. The method of using harmonic coordinate functions by Bray-Kazaras-Khuri-Stern \cite{brayHarmonic2022} provides the basic structure. However, their existence proof of appropriate harmonic coordinates, which is an essential part of their paper, does not transfer to Kähler manifolds. We instead replace them with holomorphic coordinates on $X^{2m}$ arising from the complex coordinates at infinity by the AE condition. This provides a characterization of AE Kähler manifolds and is a minor strengthening of work in Hein-LeBrun's proof of the Positive Mass Theorem \cite{heinMass2016}. Another key ingredient in proving Theorem \ref{MTheoremMassBound} is recognizing that the integral of the scalar curvature can generally only be realized as a boundary integral in two cases: in real two-dimensional manifolds using the Gauss-Bonnet Theorem as in \cite{brayHarmonic2022}, or on Kähler manifolds, which we employ here. This realization is an essential step in transforming the right-hand side of \eqref{mMassFormula} to the boundary integral in the ADM mass on the left-hand side.\\
	
Finally, as the level sets of holomorphic functions are real codimension-two manifolds, additional integrals of curvature components in the normal directions over level sets appear, contrary to \cite{brayHarmonic2022}. Furthermore, their method does not provide the full ADM mass in the Kähler case. However, since level sets of holomorphic functions on a Kähler manifold are minimal and regular for almost any level set, we employ a minimal surface argument, echoing the variational argument of Schoen and Yau \cite{schoenproof1979}. This method bounds the integrals by the missing terms in the ADM mass, completing the proof.\\

Given the integral estimate, we prove two new stability results for the Positive Mass Theorem. The first theorem applies to general sequences of AE Kähler manifolds and requires the excision of sets with vanishing boundaries in the limit. It also potentially accounts for the presence of gravitational wells known to appear in the Riemannian case:

\begin{mtheorem}[Theorem \ref{CorollaryStabNoRic}]\label{MCorollaryStabNoRic}
	Let $(X_i^{2m},g_i, J_i)$ be a sequence of AE Kähler manifolds of fixed complex dimension with nonnegative and integrable scalar curvatures, and suppose that the ADM masses $\m(g_i) \to 0$ as $i\to\infty$. Then for all $i$, there exists a domain $Z_i$ such that $M_i\setminus Z_i$ converges in the pointed Gromov-Hausdorff sense to $(\R^{2m},d_{\eucl})$:
	\begin{equation*}
		(X_i\setminus Z_i,\hat d_{g_i},p_i) \, \overset{pGH}{\longrightarrow} \, (\R^{2m}, d_{\eucl},0)
	\end{equation*}
	where $p_i\in M_i\setminus Z_i$ is any choice of base point and $\hat d_{g_i}$ is the induced length metric of $g_i$ on $M_i\setminus Z_i$. Furthermore, for any continuous $\xi\colon (0,\infty) \to (0,\infty)$ such that $\lim_{x\to 0^+} \xi(x) = 0$, we have for $i$ large enough:
	\begin{equation*}
		\area(\partial Z_i) \leq \frac{\m(g_i)^{\frac{m}{2m-2}+\frac{1}{2}}}{\xi(\m(g_i))}.
	\end{equation*}
\end{mtheorem}

For three-dimensional Riemannian manifolds, the theorem was proven by Dong and Song \cite{dongStability2023}. Some parts of their proof can be generalized from the three-dimensional case in a straightforward manner, while others require new work. Specifically, one of the essential ingredients in Dong and Song's proof was the integral inequality of Bray, Kazaras, Khuri, and Stern \cite[Theorem 1.2]{brayHarmonic2022}, and a major part of proving the theorem in the Kähler setting is demonstrating the inequality in Theorem \ref{MTheoremMassBound}. Another key point of Dong and Song's proof shows that, for any small three-dimensional cube around a level set of a certain function, there exists a two-dimensional subplane intersecting the level set nicely. This is necessary to bound distances between arbitrary points in the cube, which is essential for estimating the Gromov-Hausdorff distance. Using an iteration argument and Theorem \ref{MTheoremMassBound}, we extend this construction to find good real two-dimensional planes in $2m$-dimensional cubes.\\
	
	The second stability theorem does not require the excision of sets with vanishing boundaries, but it does require a lower bound on the Ricci curvature and uniform control at infinity:
	
	\begin{definition*}
		Given $b>0, \tau > m-1$, an AE Kähler manifold $(X^{2m},g,J)$ is $(b,\tau)$-asymptotically flat if in the given AE coordinate chart:
		\begin{equation*}
			|\del^k(g_{ij}-\delta_{ij})| \leq br^{-\tau-k}, \quad k=0,1,2.
		\end{equation*}
	\end{definition*}
	
	\begin{mtheorem}[Theorem \ref{CorollaryStabRicciBound}]\label{MCorollaryStabRicciBound}
		Fix $b>0$, $\tau>m-1 $, $\kappa>0$, and a point $p\in \R^{2m} \setminus \overline{B_1^{\eucl}(0)}$. Consider a pointed sequence $(X_i^{2m}, g_i, J_i,p_i)$ with $p_i \in X_i$ of $(b,\tau)$-asymptotically flat Kähler manifolds with fixed complex dimension, nonnegative and integrable scalar curvatures, and Ricci curvatures
		\begin{equation*}
			\Ric_{g_i} \geq -2\kappa g_i.
		\end{equation*}
		Furthermore, fix the AE charts $(\Phi_i)$ of $(X_i, g_i)$ such that $\Phi_i(p_i) = p$. If $\m(g_i)\to 0$, then $(X_i,d_{g_i},p_i)$ converges to $(\R^{2m},d_{\eucl}, p)$ in the pointed Gromov-Hausdorff sense.
	\end{mtheorem}
	
	On three-dimensional real AE manifolds, a similar statement is given by Kazaras, Khuri, and Lee \cite[Theorem 1.2]{kazarasStability2021}. The main ingredients of their proof are the integral inequality \cite[Theorem 1.2]{brayHarmonic2022} and Cheeger-Colding's Almost Splitting Theorem \cite{cheegerLower1996}. We show that Theorem \ref{MTheoremMassBound} is sufficient to apply their proof in the Kähler case. Compared to the three-dimensional case, we strengthen the statement by removing the condition $\H^2(X^{2m},\Z) = 0$, which is unnecessary here as the coordinates $z_1,\dots,z_M$ and the integral in Theorem \ref{MTheoremMassBound} are global.\\
	
	In Section \ref{SectionExamples}, we present new sequences of AE Kähler metrics on $\C^2$ satisfying the requirements of Theorem \ref{MTheoremMassBound}. As the ADM masses vanish along the sequences, we obtain three explicit families of Kähler metrics, one with and two without Ricci curvature uniformly bounded from below, for which Theorems \ref{MCorollaryStabNoRic} or \ref{MCorollaryStabRicciBound} apply. The first two families arise as a smoothing of the Burns metric on $\C^2$, while the third is generated from a given volume form. All families have the property of global GH convergence to $(\C^2,g_{\eucl})$ in the limit, and we did not find families of Kähler metrics with nonnegative scalar curvature developing a gravitational well. However, the methods provided here should be able to produce more families of Kähler metrics than we explore. Finding a family that develops a gravitational well and has vanishing mass would be very interesting for the application of Theorem \ref{MCorollaryStabNoRic}.
	
	\addtocontents{toc}{\protect\setcounter{tocdepth}{0}}
	\subsection*{Acknowledgments}
	
	I am very grateful to my advisor Hans-Joachim Hein for his support. The project was funded by the Deutsche Forschungsgemeinschaft (DFG, German Research Foundation) – Project-ID 427320536 - SFB 1442, and under Germany's Excellence Strategy EXC 2044 390685587, Mathematics Münster: Dynamics–Geometry–Structure.
	
	\addtocontents{toc}{\protect\setcounter{tocdepth}{2}}

	\section{Asymptotically Euclidean Kähler Manifolds}\label{SectionComplexAE}
	
	This section provides the background on asymptotically Euclidean (AE) Kähler manifolds necessary for proving Theorem \ref{MTheoremMassBound}. For a thorough treatment of this topic, we refer to \cite{heinMass2016} and the author's PhD thesis (in progress).
	
	\subsection{AE Kähler manifolds as Resolutions of $\C^m$}
	
	Let $(X^{2m},g,J)$ be an AE Kähler manifold with nonnegative and integrable scalar curvature: $R_g\geq 0, R_g\in L^1(X)$. General AE Riemannian manifolds can have several ends, but the work of Hein-LeBrun \cite[Proposition 2.5]{heinMass2016} shows that $X^{2m}$ has only one end. For this end, and as $X$ is AE, there exists a smooth asymptotic coordinate chart $\Phi\colon \R^{2m}\setminus \overline{B_1(o)} \to X\setminus K$ for some compact set $K\subset X$ such that, in the associated AE coordinates $(\tilde x_1,\dots,\tilde x_{2m})$, $g$ has the asymptotics
	\begin{equation}\label{MetricAEExpansion}
		|\del^k(g_{ij}-\delta_{ij})| = \OO(r^{-\tau-k}), \quad \tau>m-1, \quad k=0,1,2.
	\end{equation}
	
	To apply the methods of \cite{brayHarmonic2022}, we construct nice coordinates similar to their linear growth harmonic functions. This is the content of the following theorem and is a slight improvement on the results in \cite{heinMass2016}:
	
	\begin{theorem}\label{BiholomorphismAsympt}
		Let $(X^{2m},g,J)$ be an AE Kähler manifold. Then there exists a biholomorphism $\Psi$ from $X$ onto a resolution of finitely many points of $\C^m$. Furthermore, the standard coordinates $z_1,\dots,z_m$ on $\C^m$ pulled back to the resolution space preserve the AE asymptotics of $g$:
		\begin{equation*}
			|\del^k((\Psi_* g)_{i\overline{j}}-\delta_{i\overline{j}})| = \OO(|z|^{-\tau-k}), \quad k=0,1,2,
		\end{equation*}
		for some $\tau>m-1$.
	\end{theorem}
	
	The theorem is proved as follows: by \cite{heinMass2016}, any AE Kähler manifold is biholomorphic to a resolution of finitely many points on $\C^m$. Furthermore, picking an AE Kähler manifold $(X,g)$ biholomorphic to some resolution space $\widehat{\C^m}$ of $\C^m$, one can pick the biholomorphism $\Psi \colon X \to \widehat{\C^m}$ such that $\Psi_*g$ preserves the asymptotics \eqref{MetricAEExpansion} in the natural coordinates $(z_1,\dots,z_m)$ on $\C^m$ pulled back to $\widehat{\C^m}$. The real and imaginary components of $(z_1,\dots,z_m)$ (up to reordering) are asymptotic to the AE coordinates $(\tilde x_1,\dots,\tilde x_{2m})$ from \eqref{MetricAEExpansion}.

	\subsection{The ADM Mass of Kähler Manifolds}

	Given an AE Kähler manifold $(X^{2m},g,J)$ with integrable scalar curvature, $(X,g)$ has a well-defined ADM mass. Furthermore, the Kähler condition simplifies the expression of the mass to an integral of the scalar curvature plus a topological term. Denote by $\clubsuit \colon \H_{\text{dR}}^2(X)\to \H^2_c(X)$ the inverse of the map $\H^2_c(X) \to \H_{\text{dR}}^2(X)$ induced by the inclusion of compact differential forms into all differential forms. We have the following result regarding the ADM mass of an AE Kähler manifold:
	
	\begin{theorem}[{\cite[Theorem C]{heinMass2016}}]\label{MassTheorem}
		If $(X^{2m},g,J)$ is an AE Kähler manifold with scalar curvature $R_g\in L^1(X)$, the ADM mass of $X$ equals
		\begin{equation}\label{MassFormula}
			\m(g) = -\frac{\gen{\clubsuit(c_1(X)),[\omega^{m-1}]}}{(2m-1)\pi^{m-1}} + \frac{(m-1)!}{4(2m-1)\pi^m} \int_X R_g \dvol_g,
		\end{equation}
		where $c_1(X)$ is the first Chern class of $X$, $[\omega] \in \H_{\text{dR}}^2(X)$ is the Kähler class of $g$, and $\gen{\cdot,\cdot}$ is the duality pairing of $\H^2_c(X)$ and $\H_{\text{dR}}^{2m-2}(X)$.
	\end{theorem}
	
	We stress that the pairing
	\begin{equation}\label{MassFirstTerm}
		- \frac{\gen{\clubsuit(c_1(X)),[\omega^{m-1}]}}{(2m-1)\pi^{m-1}} \geq 0
	\end{equation}
	is nonnegative for AE Kähler manifolds. This is important for inequalities appearing in the proof of Lemma \ref{MassScalarCurvIneqRightHand}. Let us briefly recall the reason for \eqref{MassFirstTerm}: take $\Psi \colon X \to \C^m$ defined via the coordinates $(z_1,\dots,z_m)$ and define the holomorphic form $\Omega = dz_1\wedge \cdots \wedge dz_m$. Because $\Omega\in\H^0(X, K_X)$ is a section of the canonical bundle $K_X$, the zero section $E = \Omega^{-1}(0)$ is the generating divisor of $K_X$ and hence the Poincare dual of $\clubsuit (C_1(K_X)) = - \clubsuit (C_1(X,J))$. \eqref{MassFirstTerm} is thus the evaluation of the area form on $E$, which is nonnegative. Furthermore, if $\gen{\clubsuit c_1(X),[\omega]^{m-1}} = 0$, then $E$ is empty and $\Psi$ is a global biholomorphism.

	\section{Mass Inequality via Level Sets and The Positive Mass Theorem}\label{SectionMassIneq}
	
	\subsection{Scalar Curvature and Level Sets}

	This section derives a lower bound on the ADM mass in terms of the holomorphic coordinates on AE Kähler manifolds from Theorem \ref{BiholomorphismAsympt} and the scalar curvature. Part of the proof is inspired by \cite{brayHarmonic2022}, and we use the results from Section \ref{SectionComplexAE} to modify the arguments from three-dimensional Riemannian manifolds to AE Kähler manifolds by instead employing the complex coarea formula. A new complication, giving rise to additional curvature terms, appears in the Kähler case and is dealt with in Lemma \ref{LemmaCurvatureRestTerm}.\\
	
	Let $(z_1,\dots,z_m)$ be the coordinates coming from Theorem \ref{BiholomorphismAsympt}. We only consider $z_1$ in the following computations, but the results apply to all $z_i$. Write the real and imaginary components of $(z_1,\dots,z_m)$ as
	\begin{equation*}
		z_1 = x_1+ix_2,\dots, z_m = x_{2m-1}+ix_{2m}.
	\end{equation*}
	
	We need the complex version of the coarea formula. For any complex-valued function $u$, $\Re u$ and $\Im u$ denote the real and imaginary components of $u$, respectively.
	
	\begin{prop}
		Let $f\in L^1(X,g)$ and $u\colon X \to \C$ be holomorphic. Then
		\begin{equation}\label{CoareaComplex}
			\int_{X} f\; |\nabla \Re u|^2 \dvol_g = \int_{\C} \int_{u^{-1}(s)} f(s,z_2,\dots,z_m) \dA_g ds_1 ds_2,
		\end{equation}
		where $\dA_g$ is the area form on $u^{-1}(s)$ and $s = s_1+is_2$ is the complex coordinate on $\C$.
	\end{prop}
		\begin{proof}
			Since $|\nabla \Re u|= |\nabla \Im u|$ and $\nabla \Re u \perp \nabla \Im u$ because $g$ is Kähler and by the Cauchy-Riemann equations, the above formula is obtained by applying the real coarea formula twice.
		\end{proof}
		
		From now on, let $s = s_1+is_2\in \C$ with level set $\Sigma_s \coloneqq \{z_1 = s\}$. If $s$ is fixed, we write $\Sigma$. As we later integrate over all $s\in \C$, Sard's Theorem implies that, without loss of generality, we can assume that $s$ is a regular value of $x_1$ and $x_2$. The first step is to obtain an estimate involving the scalar curvature on $X$ and the intrinsic scalar curvature on $\Sigma$, denoted by $R_{g}$ and $R_{\Sigma}$ respectively, as well as the Ricci curvature on $X$ and the second fundamental form $\sff$.
		
		\begin{lemma}\label{LemmaGauss}
			Given the AE Kähler manifold $(X,g)$ and $z_1$ as above with regular value $s\in \C$, then
			\begin{align*}
				\frac{1}{2}|\nabla x_1|^2 (R_g - R_\Sigma - |\sff|^2 + 2\gen{\Rm_{X}(\hat x_1,\hat x_2)\hat x_2,\hat x_1}) &= \Ric(\nabla x_1, \nabla x_1) + \Ric(\nabla x_2, \nabla x_2)
			\end{align*}
		\end{lemma}
		
		\begin{proof}
			Define $\hat x_i = \nabla x_i/|\nabla x_i|$ for $i=1,2$ as the two unit normal vectors of $\Sigma\coloneqq \Sigma_s$. $\Sigma$ has a (real) normal bundle $\Span_\R\{\hat x_1,\hat x_2\}$ as $\Sigma$ is a level set of $z_1$. Recall the Gauss formula for $p\in \Sigma$ and $X,Y,Z,W \in T_p(\Sigma)$:
			\begin{equation*}
				\gen{\Rm_{X}(X,Y)Z,W} = \gen{\Rm_\Sigma(X,Y)Z,W} + \gen{s(X,Z),s(Y,W)} - \gen{s(Y,Z),s(X,W)}
			\end{equation*}
			with $\sff(X,Y) \coloneqq (\nabla_X Y)^\perp$ being the second fundamental form on $\Sigma$. Taking the trace twice over the coordinates in $\Sigma$ yields:
			\begin{equation}\label{RicciEq}
				\begin{aligned}
					R_g - 2\Ric(\hat x_1,\hat x_1) - 2\Ric(\hat x_2, \hat x_2) + 2\gen{\Rm_{X}(\hat x_1,\hat x_2)\hat x_2,\hat x_1}= R_\Sigma + |\sff|^2 - \morm{H}^2.
				\end{aligned}
			\end{equation}
			Here, $H$ is the mean curvature of $\Sigma$ and vanishes since $\Sigma$ is a complex submanifold of a Kähler manifold. Multiplying by $\frac{1}{2}|\nabla x_1|^2$ and rearranging finishes the proof.
		\end{proof}

	The next step is to integrate Lemma \ref{LemmaGauss} over a domain $\Omega \subset X$ and apply the complex coarea formula \eqref{CoareaComplex}.
	
	\begin{lemma}\label{Lemma2}
		Let $\Omega\subset M$ be a precompact domain with a connected boundary $\partial \Omega \subset M$. Then
		\begin{equation}\label{IntegralLemma2}
			\begin{aligned}
				&2\int_{\partial\Omega} \morm{\nabla x_1} \del_\nu \morm{\nabla x_1} \dA_g + \int_{s\in\C}\int_{\Sigma_s\cap \Omega} R_\Sigma \dA_g ds_1ds_2\\ =& \int_\Omega |\nabla^2 x_1|^2 +|\nabla^2 x_2|^2 + \morm{\nabla x_1}^2 R_g \dvol_g \\
				& + \int_{s\in\C}\int_{\Sigma_s\cap \Omega}- |\sff|^2 -\tr_\Sigma (\gen{\Rm_X(\hat x_1,\cdot)\cdot,\hat x_1} + \gen{\Rm_X(\hat x_2,\cdot)\cdot,\hat x_2}) \dA_g ds_1ds_2,
			\end{aligned}
		\end{equation}
		where $\dA_g$ is either the area form of $\partial \Omega$ or $\Sigma_s$.
	\end{lemma}
	
	\begin{proof}
		By definition of the Ricci curvature and as it is compatible with $J$, then
		\begin{equation}\label{RicciSectionalRelation}
			\begin{aligned}
				\Ric(\hat x_1,\hat x_1) = \Ric(\hat x_2,\hat x_2) &= \gen{\Rm_{X}(\hat x_1, \hat x_2)\hat x_2, \hat x_1} + \tr_\Sigma \gen{\Rm_X(\hat x_1,\cdot)\cdot,\hat x_1}\\
				&=\gen{\Rm_{X}(\hat x_1, \hat x_2)\hat x_2, \hat x_1} + \tr_\Sigma \gen{\Rm_X(\hat x_2,\cdot)\cdot,\hat x_2}.\\
			\end{aligned}
		\end{equation}
		
		Next, as $x_1$ and $x_2$ are the real and imaginary parts of a holomorphic function, they are harmonic for any Kähler metric. Thus, according to Bochner's formula:
		\begin{equation*}
			\frac{1}{2} \Delta \morm{\nabla x_i}^2 = |\nabla^2 x_i|^2 + \Ric(\nabla x_i, \nabla x_i), \qquad i=1,2,
		\end{equation*}
		so Lemma \ref{LemmaGauss} and \eqref{RicciSectionalRelation} implies
		\begin{align*}
			&\frac{1}{2} \Delta \morm{\nabla x_1}^2 + \frac{1}{2} \Delta \morm{\nabla x_2}^2\\
			=& \, |\nabla^2 x_1|^2 + |\nabla^2 x_2|^2 + \Ric(\nabla x_1, \nabla x_1) + \Ric(\nabla x_2, \nabla x_2)\\
			=&|\nabla^2 x_1|^2 + |\nabla^2 x_2|^2 \\
			&+ \morm{\nabla x_1}^2 \left(R_g - R_\Sigma - |\sff|^2 - \tr_\Sigma \gen{\Rm_X(\hat x_1,\cdot)\cdot,\hat x_1} - \tr_\Sigma \gen{\Rm_X(\hat x_2,\cdot)\cdot,\hat x_2}\right).
		\end{align*}
		Moving $\morm{\nabla x_1}^2 R_\Sigma$ to the other side and integrating over $\Omega$:
		\begin{align*}
			&\int_\Omega \frac{1}{2} \Delta \morm{\nabla x_1}^2 + \frac{1}{2} \Delta \morm{\nabla x_2}^2 + \morm{\nabla x_1}^2   R_\Sigma\dvol_g\\
			=& \int_\Omega |\nabla^2 x_1|^2 + |\nabla^2 x_2|^2 + \morm{\nabla x_1}^2 R_g\\
			&- \morm{\nabla x_1}^2 \left(|\sff|^2 +\tr_\Sigma \gen{\Rm_X(\hat x_1,\cdot)\cdot,\hat x_1} + \tr_\Sigma \gen{\Rm_X(\hat x_2,\cdot)\cdot,\hat x_2})\right) \dvol_g 
		\end{align*}
		On the left-hand side, we use the divergence theorem and identify $|\nabla x_1| = |\nabla x_2|$ to obtain
		\begin{equation*}
			\int_\Omega \frac{1}{2} \Delta \morm{\nabla x_1}^2 + \frac{1}{2} \Delta \morm{\nabla x_2}^2\dvol_g = 	2\int_{\partial\Omega} \morm{\nabla x_1} \del_\nu \morm{\nabla x_1} \dA_g,
		\end{equation*}
		where $d A_g$ is the area form of $g$. 
		
		Finally, use the complex coarea formula (Proposition \ref{CoareaComplex}) to obtain  the lemma:
		\begin{align*}
			&2\int_{\partial\Omega} \morm{\nabla x_1}\del_\nu \morm{\nabla x_1} \dA_g + \int_{s\in\C}\int_{\Sigma_s\cap \Omega} R_\Sigma \dA_g ds_1ds_2\\ =& \int_\Omega |\nabla^2 x_1|^2 +|\nabla^2 x_2|^2 + \morm{\nabla x_1}^2 R_g \dvol_g \\
			& + \int_{s\in\C}\int_{\Sigma_s\cap \Omega}- |\sff|^2 -\tr_\Sigma (\gen{\Rm_X(\hat x_1,\cdot)\cdot,\hat x_1} + \gen{\Rm_X(\hat x_2,\cdot)\cdot,\hat x_2}) \dA_g ds_1ds_2.
		\end{align*}
	\end{proof}

	\subsection{Obtaining the Mass from an Integral Inequality}\label{SectMassIneq}
	
To proceed with the proof of Theorem \ref{MTheoremMassBound}, we choose a nice domain $\Omega = \Omega_L$ that depends on $L$ and let $L\to\infty$ in Lemma \ref{Lemma2}. This will give us the ADM mass. The domain is a complex version of the domain chosen in \cite[Eq. (6.1)]{brayHarmonic2022}. We define $\Omega_L\subset M$ as follows:
\begin{equation*}
	\Omega_L \coloneqq \left\{|x_1|\leq L, |x_2|\leq L \right\} \times \left\{|(x_3,\dots,x_{2m})| \leq L \right\},
\end{equation*}
where $|\cdot|$ represents the natural Euclidean distance from $0$ with respect to the coordinate chart $(x_1,\dots,x_{2m})$.

In turn, the boundary of $\Omega_L$ is given by:
\begin{align*}
	\partial \Omega_L \eqqcolon C_L &= \left\{x_1 = \pm L, |x_2|\leq L \right\} \times \left\{|(x_3,\dots,x_{2m})| \leq L \right\}\\
	&\cup \left\{|x_1|\leq L, x_2= \pm L \right\} \times \left\{|(x_3,\dots,x_{2m})| \leq L \right\} \\
	&\cup\left\{|x_1|\leq L, |x_2|\leq L \right\} \times \left\{|(x_3,\dots,x_{2m})| = L \right\}.
\end{align*}
	\begin{equation}\label{BoundaryComp}
		\begin{aligned}
			D_{1,L}^\pm  &\coloneqq \left\{x_1 =  \pm L, |x_2|\leq L  \right\} \times \left\{|(x_3,\dots,x_{2m})| \leq L \right\},\quad \nu = \pm\del_1,\\
			D_{2,L}^\pm  &\coloneqq \left\{|x_1|\leq L, x_2= \pm L  \right\} \times \left\{|(x_3,\dots,x_{2m})| \leq L \right\},\quad \nu = \pm\del_2, \\
			T_L &\coloneqq \left\{|x_1|\leq L, |x_2|\leq L  \right\} \times \left\{|(x_3,\dots,x_{2m})| = L \right\}, \quad \nu =\frac{1}{L} \sum_{l=3}^{2m} x_l\del_l.
		\end{aligned}
	\end{equation}
	The boundaries above only overlap on a subset of measure zero. The $\nu$'s are the outer normals of the boundaries for $g$. As $g$ is AE, for a large enough $L$, the normals of $g$ and $g_{\eucl}$ coincide up to an error of order $\OO(|x|^{-\tau})$. The area form $\dA_g$ on $\partial \Omega_L$ can similarly be replaced by the Euclidean area form $\dA_{\eucl}$ up to an error of order $\OO(|x|^{-\tau})$. In the three-dimensional case, a similar situation appears in \cite{brayHarmonic2022}. However, as the level sets are now real codimension-two manifolds, additional terms appear (see Lemma \ref{LemmaCurvatureRestTerm}).
	
	\begin{lemma}\label{Lemma3} The first term of the left-hand side of \eqref{IntegralLemma2} is 
		\begin{equation*}
			\begin{aligned}
				&2\int_{C_L} \morm{\nabla x_1} \del_\nu \morm{\nabla x_1} \dA_g\\ =& \int_{D_{1,L}^+} \sum_{k=1}^{2m}\left( g_{1k,k}  - g_{kk,1}\right) \dA_{\eucl}  +\int_{D_{2,L}^+}  \sum_{k=1}^{2m}\left( g_{2k,k}  - g_{kk,2}\right) \dA_{\eucl} 
				\\ &+\int_{D_{1,L}^-}  \sum_{k=1}^{2m}\left(- g_{1k,k}  + g_{kk,1}\right) \dA_{\eucl}  + \int_{D_{2,L}^-} \sum_{k=1}^{2m}\left(- g_{2k,k}  + g_{kk,2}\right) \dA_{\eucl}\\
				&+\frac{1}{L} \int_{T_L} \sum_{i=1}^2  \sum_{k=3}^{2m}(g_{ik,i} -  \frac{1}{2} g_{ii,k}) x_k \dA_{\eucl} +  \OO(L^{-2\tau-2+2m}) .
			\end{aligned}
		\end{equation*}
	\end{lemma}

	\begin{proof}
		First, we see that 
		\begin{equation*}
			\morm{\nabla x_1} = 1 + \OO(|x|^{-\tau}),
		\end{equation*}
		so we can ignore the first factor of $\morm{\nabla x_1}$ in $\morm{\nabla x_1} \del_\nu \morm{\nabla x_1}$ to highest order. Next,
		\begin{equation*}
			\nabla x_1 = g^{1l}\del_l, \qquad \nabla|\nabla x_1| = \nabla(g^{11})^{1/2} = -\frac{1}{2}\nabla g_{11} + \OO(|x|^{-2\tau-1}),
		\end{equation*}
		where the second equation is due to the decay rates. 
		Due to the normal vectors in \eqref{BoundaryComp} and $|\nabla x_1| = 1+ \OO(|x|^{-\tau})$, then
		
		\begin{equation}\label{BoundaryLaplace1}
			\begin{aligned}
				2\int_{C_L} \morm{\nabla x_1} \del_\nu \morm{\nabla x_1} \dA_g =& 2\int_{C_L} \del_\nu \morm{\nabla x_1}\dA_g + \OO(L^{-2\tau-2+2m})\\ =&
				-\int_{D_{1,L}^+}  g_{11,1} \dA_{\eucl}  -\int_{D_{2,L}^+}  g_{11,2} \dA_{\eucl} 
				\\ &
				+\int_{D_{1,L}^-}  g_{11,1} \dA_{\eucl}  + \int_{D_{2,L}^-} g_{11,2} \dA_{\eucl}\\
				&-\frac{1}{L} \int_{T_L} \sum_{k=3} g_{11,k} x_k \dA_{\eucl} +  \OO(L^{-2\tau-2+2m)}).
			\end{aligned}
		\end{equation}
		Let two repeated indices indicate summation over the coordinates. As $x_1$ is harmonic, then
		\begin{align*}
			0 
			&= \sum_{k=1}^{2m}\left(\frac{1}{2} g_{kk,1} - g_{1k,k}\right) + \OO(|x|^{-2\tau-1}).
		\end{align*}
		Rearranging \eqref{BoundaryLaplace1} and doing the same computation for $g_{22}$, then
		\begin{align*}
			g_{11,1} = \sum_{k=2}^{2m}\left(-2 g_{1k,k}  + g_{kk,1}\right) + \OO(|x|^{-2\tau-1}), \\ g_{22,2} = \sum_{k=1, k\neq 2}^{2m}\left(-2 g_{2k,k}  + g_{kk,2}\right) + \OO(|x|^{-2\tau-1}).
		\end{align*}
		Writing $g_{11} = g_{22} = \frac{1}{2}(g_{11} + g_{22})$ ($g_{11} = g_{22}$ due to compatibility with $J$), \eqref{BoundaryLaplace1} becomes
		\begin{equation}\label{BoundaryLaplacian2}
			\begin{aligned}\sloppy
				&2\int_{C_L} \morm{\nabla x_1} \del_\nu \morm{\nabla x_1} \dA_g \\ 
				=& \int_{D_{1,L}^+} \sum_{k=1}^{2m}\left( g_{1k,k}  - g_{kk,1}\right) \dA_{\eucl}  +\int_{D_{2,L}^+}  \sum_{k=1}^{2m}\left( g_{2k,k}  - g_{kk,2}\right) \dA_{\eucl} 
				\\ &
				+\int_{D_{1,L}^-}  \sum_{k=1}^{2m}\left(- g_{1k,k}  + g_{kk,1}\right) \dA_{\eucl}  + \int_{D_{2,L}^-} \sum_{k=1}^{2m}\left(- g_{2k,k}  + g_{kk,2}\right) \dA_{\eucl}\\
				&\frac{1}{L} \int_{T_L} \sum_{i=1}^2\sum_{k=3} -\frac{1}{2} g_{ii,k} x_k\dA_{\eucl} \\
				&+\int_{D_{1,L}^+} \sum_{k=2}^{2m} g_{1k,k} \dA_{\eucl}  +\int_{D_{2,L}^+}  \sum_{k=1, k\neq 2}^{2m} g_{2k,k}  \dA_{\eucl} 
				\\ &
				-\int_{D_{1,L}^-}  \sum_{k=2}^{2m} g_{1k,k}   \dA_{\eucl}  - \int_{D_{2,L}^-} \sum_{k=1, k\neq 2}^{2m} g_{2k,k} \dA_{\eucl} +  \OO(L^{-2\tau-2+2m}).
			\end{aligned}
		\end{equation}
		The two first lines after the last equality represent all the mass components on $D_{1,L}^{\pm}$ and $D_{2,L}^{\pm}$. We rewrite the two last lines as integrals on $T_L$ using the divergence theorem as follows: first, we express them as integrals on $\partial D_{1,L}^{\pm}$ and $\partial D_{2,L}^{\pm}$. Then, we apply the divergence theorem again to obtain integrals on $T_L$. More concretely, write
		\begin{align*}
			\sum_{k=2}^{2m} g_{1k,k} &= \Div_{\eucl} (0,g_{12},g_{13},\dots g_{1(2m)}),\\
			\sum_{k=1,k\neq 2}^{2m} g_{2k,k} &= \Div_{\eucl} (g_{21},0,g_{23},\dots g_{2(2m)}).
		\end{align*}
		The boundaries equal
		\begin{equation*}
			\begin{aligned}
				\partial D_{1,L}^\pm =& \left\{x_1 =  \pm L, |x_2|= L  \right\} \times \left\{|(x_3,\dots,x_{2m})| \leq L \right\}\\
				&\cup \left\{x_1 =  \pm L, |x_2|\leq L  \right\} \times \left\{|(x_3,\dots,x_{2m})| = L \right\}.\\
				\partial D_{2,L}^\pm =& \left\{|x_1| = L, x_2 =\pm L  \right\} \times \left\{|(x_3,\dots,x_{2m})| \leq L \right\}\\
				&\cup \left\{|x_1| \leq L, x_2 =\pm L  \right\} \times \left\{|(x_3,\dots,x_{2m})| = L \right\}.\\
			\end{aligned}
		\end{equation*}
		Using the divergence theorem to write $\int_{D_{1,L}^\pm}$ as a boundary integral, we do not get contributions from the components containing $\{|x_2|= L\}$ as $g_{12} =0$ by the Kähler condition. Same for $\int_{D_{2,L}^\pm}$ and $|x_1| = L$. Thus:
		\begin{equation}\label{TwoIntegrals}
			\begin{aligned}
				\int_{D_{1,L}^\pm}  \pm& \left(\sum_{k=3}^{2m} g_{1k,k} \right) \dA_{\eucl} \\&= \int_{\left\{x_1 =  \pm L, |x_2|\leq L  \right\} \times \left\{|(x_3,\dots,x_{2m})| = L \right\}}  \pm\frac{1}{L}\left(\sum_{k=3}^{2m} g_{1k}x_k \right) \dA_{\eucl},\\
				\int_{D_{2,L}^\pm}  \pm&\left(\sum_{k=3}^{2m} g_{2k,k} \right) \dA_{\eucl}\\ &= \int_{\left\{|x_1| \leq L, x_2 = \pm L  \right\} \times \left\{|(x_3,\dots,x_{2m})| = L \right\}}  \pm\frac{1}{L}\left(\sum_{k=3}^{2m} g_{2k}x_k \right) \dA_{\eucl}.
			\end{aligned}
		\end{equation}
		However, consider the integral:
		\begin{equation}\label{DivergenceCylinderPart}
			\begin{aligned}
				&\int_{T_L} \Div_{\eucl} \left( \left(\frac{1}{L}\sum_{k=3}^{2m} g_{1k}x_k\right) , \left(\frac{1}{L}\sum_{k=3}^{2m}g_{2k}x_k\right),0,\cdots,0 \right) \dA_{\eucl}\\ 
				=& \frac{1}{L}\int_{T_L} \left(\sum_{k=3}^{2m} g_{1k,1}x_k + \sum_{k=3}^{2m}g_{2k,2}x_k \right)  \dA_{\eucl}.
			\end{aligned}
		\end{equation}
		Using the divergence theorem on \eqref{DivergenceCylinderPart} we obtain a sum of the four integrals on the right-hand sides of \eqref{TwoIntegrals}. Inserting this into \eqref{BoundaryLaplacian2} shows that
		\begin{equation}\label{BoundaryLaplacian3}
			\begin{aligned}
				&2\int_{C_L} \morm{\nabla x_1} \del_\nu \morm{\nabla x_1} \dA_g\\ =& \int_{D_{1,L}^+} \sum_{k=1}^{2m}\left( g_{1k,k}  - g_{kk,1}\right) \dA_{\eucl}  +\int_{D_{2,L}^+}  \sum_{k=1}^{2m}\left( g_{2k,k}  - g_{kk,2}\right) \dA_{\eucl} 
				\\ &+\int_{D_{1,L}^-}  \sum_{k=1}^{2m}\left(- g_{1k,k}  + g_{kk,1}\right) \dA_{\eucl}  + \int_{D_{2,L}^-} \sum_{k=1}^{2m}\left(- g_{2k,k}  + g_{kk,2}\right) \dA_{\eucl}\\
				&+\frac{1}{L} \int_{T_L} \sum_{i=1}^2  \sum_{k=3}^{2m}(g_{ik,i} -  \frac{1}{2} g_{ii,k}) x_k \dA_{\eucl} +  \OO(L^{-2\tau-2+2m}).
			\end{aligned}
		\end{equation}
	\end{proof}
	
	We remark that \eqref{BoundaryLaplacian3} contains many terms of the mass in the limit $L\to\infty$. However, there is an unwanted factor of $\frac{1}{2}$ in the last line of the last terms. This will be remedied in Lemma \ref{LemmaCurvatureRestTerm}.\\
	
	On three-dimensional Riemannian manifolds, a variation of the next lemma was proved in \cite{brayHarmonic2022} using the Gauss-Bonnet Theorem, rewriting the integral of the scalar curvature as a boundary integral. In higher dimensions, this is generally not possible. However, this method is viable on Kähler manifolds by Theorem \ref{MassTheorem} as the level sets are complex submanifolds of a Kähler manifold.
	
	\begin{lemma}\label{MassScalarCurvIneqRightHand}
		The second term on the left-hand side of \eqref{IntegralLemma2} is bounded from above by
		\begin{equation*}
			\int_{s\in\C} \int_{\Sigma_s} R_{\Sigma_s} \dA_g ds \leq \lim_{L\to\infty} \int_{T_L} \sum_{k,l=3}^{2m} (g_{kl,k} - g_{kk,l})x_l \dA_{\eucl}.
		\end{equation*}
	\end{lemma}
	
	\begin{proof}
		By Theorem \ref{MassFormula} and because $\Sigma_s$ is a complex AE Kähler submanifold for almost all $s\in \C$, we know that
		\begin{align*}
			\int_{\Sigma_s} R_{\Sigma_s} \dA_g =& \frac{4\pi\gen{\clubsuit(c_1(\Sigma_s)),[\omega^{m-2}]}}{(m-2)!} \\ &+ \lim_{L\to\infty} \int_{\partial (\Omega_L \cap \Sigma_s)} \sum_{k,l=3}^{2m} (g_{kl,k} - g_{kk,l})\nu_l \dA_{\eucl}.
		\end{align*}
		The pairing $\gen{\clubsuit(c_1(\Sigma_s)),[\omega^{m-2}]} \leq 0$ (see \eqref{MassFirstTerm}) is negative. Furthermore, we easily see that
		\begin{equation*}
			\partial( \Omega_L \cap \Sigma_s) = \{x_1+ix_2 = s \} \times \left\{|(x_3,\dots,x_{2m})| = L \right\}
		\end{equation*}
		assuming $|\Re s|\leq L$ and $|\Im s| \leq L$, otherwise the set is empty. Hence, it follows that
		\begin{align*}
			&\lim_{L\to\infty}\int_{s\in\C} \int_{\Omega_L \cap \Sigma_s} R_{\Sigma_s} \dA_g ds_1ds_2 \\
			\leq& \lim_{L\to\infty} \int_{s\in\C} \int_{\partial (\Omega_L \cap \Sigma_s)} \sum_{k,l=3}^{2m} (g_{kl,k} - g_{kk,l})\nu_l \dA_{\eucl} ds_1ds_2\\
			=&\lim_{L\to\infty}  \frac{1}{L}\int_{T_L} \sum_{k,l=3}^{2m} (g_{kl,k} - g_{kk,l})x_l \dA_{\eucl}.
		\end{align*}
	\end{proof}

	The final step in proving the mass inequality is to show that the double integral on the right-hand side of \eqref{IntegralLemma2} is bounded below by some Hessians and a boundary term as $L\to \infty$. The boundary term then exactly remedies the factor of $\frac{1}{2}$ in Lemma \ref{Lemma3}, and neither the double integral nor the factor of $\frac{1}{2}$ is apparent in \cite{brayHarmonic2022}. The double integrals appear as we consider level sets of holomorphic functions, i.e. real codimension-two manifolds, and so curvature terms only involving normal components can exist in the complex case.\\
	
	To remedy the double integral and factor of $\frac{1}{2}$, we use the formula for the second variation of minimal submanifolds to rewrite $\gen{\Rm_{X}(\hat x_1, \hat x_2)\hat x_2, \hat x_1}$ via terms appearing in this expression. As $\Sigma$ is a complex submanifold of a Kähler manifold, it is absolutely area-minimizing with respect to compactly supported perturbations by the Wirtinger inequality. As we will see, it is also minimizing with respect to certain variations without compact support, and we will choose a good variation proving the desired inequality of Lemma \ref{LemmaCurvatureRestTerm}. This idea to use the variation of minimal surfaces is an adaptation of the minimal surface argument for negative mass appearing in \cite{schoenproof1979}. Now, let $\Sigma(t)$ be a variation of $\Sigma$ with a normal variational vector field $F_t$. By the variational formula for submanifolds \cite[Eq. (1.142) and (1.143)]{coldingCourse2011}:
	\begin{equation}\label{VariationEq}
		\begin{aligned}
			\frac{d^2}{dt^2} \area(\Sigma(t))\Big|_{t=0} =& \int_{\Sigma} - |\gen{s(\cdot,\cdot),F_t}|^2 + |\nabla^\perp_\Sigma F_t|^2 - \tr_\Sigma \gen{\Rm_{X}(F_t,\cdot)\cdot, F_t}\\
			&+ \int_{\partial\Sigma} \gen{\nabla_{F_t} F_t, \nu}d A_g, 
		\end{aligned}
	\end{equation}
	where $\nabla^\perp_\Sigma$ is the gradient restricted to $\Sigma$ and projected onto the normal components, and $\nu$ is the outward pointing unit normal of $\partial \Sigma$.

	\begin{lemma}\label{LemmaCurvatureRestTerm}
		The final term of \eqref{IntegralLemma2} is bounded by below by
		\begin{align*}
			\int_{s\in\C}\int_{\Omega_L \cap \Sigma_s}&- |\sff|^2 - \tr_\Sigma (\gen{\Rm_X(\hat x_1,\cdot)\cdot,\hat x_1} + \gen{\Rm_X(\hat x_2,\cdot)\cdot,\hat x_2}) \, \dA_g ds_1ds_2\\
			\geq& -\int_{\Omega_L} \frac{1}{2}|\nabla^2 x_1|^2 + \frac{1}{2}|\nabla^2 x_2|^2 \dvol_{g} + \frac{1}{2L}\int_{T_L} \sum_{i=1}^2\sum_{k=3}^{2m}  g_{ii,k} x_k \dA_{\eucl}\\ &+ \OO(L^{-2\tau-2+2m}).
		\end{align*}
	\end{lemma}
	
	\begin{proof}
		On $\Sigma_L$, the double integral on the right-hand side of \eqref{IntegralLemma2} is
		\begin{equation}\label{IntegralLeftover}
			\int_{s\in\C}\int_{\Omega_L \cap \Sigma_s}- |\sff|^2 - \tr_\Sigma (\gen{\Rm_X(\hat x_1,\cdot)\cdot,\hat x_1} + \gen{\Rm_X(\hat x_2,\cdot)\cdot,\hat x_2}) \, \dA_g ds_1ds_2.
		\end{equation}
		We want to choose a variational vector field $F_t$ such that, when inserted into \eqref{VariationEq}, we obtain the terms of \eqref{IntegralLeftover}. To start, let $F_t = \frac{\hat x_i}{|\nabla x_i|}$ be the variational vector field for $i=1,2$. Choosing $i=1$, the flow of $\Sigma \coloneqq \Sigma_{s_1+is_2}$ along $F_t$ gives new level sets $\Sigma_{(s_1+t)+is_2}$, which will be desirable. However, using this variation field in \eqref{VariationEq} gives an extra factor of $\frac{1}{|\nabla x_1|^2}$ in the integral compared to \eqref{IntegralLeftover}. Therefore, choose the following variational vector field that interpolates between $\hat x_1$ and $\frac{\hat x_1}{|\nabla x_1|}$ as follows: let $u^2 = |z_2|^2 + \dots + |z_m|^2$ and define (a smoothing of) the continuous cutoff function:
		\begin{equation*}
			\chi_L(z_1,\dots, z_m) = \begin{cases}
				1 & u \leq \frac{L}{4}\\
				\frac{4}{L}(\frac{1}{|\nabla x_1|}-1)u + (2-\frac{1}{|\nabla x_1|}) & \frac{L}{4} < u < \frac{L}{2}\\
				\frac{1}{|\nabla x_1|} & u \geq \frac{L}{2}.
			\end{cases}
		\end{equation*}
		Let $F_t \coloneqq \chi_L \hat x_1$. Define $a>0$ to be small enough such that $E_a \coloneqq \{\Sigma_{(s_1+t)+is_2} \cap \Omega_L \mid t\in [0,a]\}$ is a smooth $(2m-1)$-dimensional submanifold consisting of regular level sets. Let $\Sigma^t_{s_1+is_2}$ be the flow of $\Sigma_{s_1+is_2}$ by $F_t$.
		\begin{equation*}
			E'_a \coloneqq \{\Sigma^t_{s_1+is_2} \cap \Omega_L \mid t\in [0,a] \}.
		\end{equation*}
		By definition of $F_t$, then
		\begin{equation*}
			\Sigma_{(s_1+t)+is_2} \cap \left(\Omega_L \setminus \Omega_{\frac{L}{2}} \right) = \Sigma_{s_1+is_2}^t \cap \left(\Omega_L \setminus \Omega_{\frac{L}{2}} \right).
		\end{equation*}
		In a neighborhood of the boundary $\partial E_a' \cap \partial \Omega_L$, which is along the flowlines, the level sets are complex submanifolds. Let $p\in \partial E_a' \cap \partial \Omega_L$ and $Y\in T_p(\Sigma_{s_1+is_2}^t)$. Then, $Y \perp \hat x_1$. Since the level sets are complex submanifolds close to $\partial E_a' \cap \partial \Omega_L$, it follows that $Y \perp \hat x_2$ as well. By Stokes' Theorem, it follows that
		\begin{equation}\label{IntegrationStokesTrick}
			\begin{aligned}
				0 &= \int_{E_a'} \frac{1}{(m-1)!}d(\omega^{m-1})\\ &= -\int_{\Sigma_{s_1+is_2}} \frac{1}{(m-1)!}\omega^{m-1} + \int_{\Sigma_{s_1+is_2}^t}\frac{1}{(m-1)!}\omega^{m-1} \\
				&\leq -\area_g (\Sigma_{s_1+is_2}) + \area_g(\Sigma_{s_1+is_2}^t),
			\end{aligned}
		\end{equation}
		With equality if and only if $\Sigma_{(s_1)+is_2}^t$ is a complex submanifold, as proven by the Wirtinger inequality. Since the evaluation of the integral along $\partial E_a'\cap \partial \Omega_L$ contains a factor of
		\begin{equation*}
		\omega(Y,\hat x_1) = \gen{Y,\hat x_2} = 0,
		\end{equation*}
		the integral along the flowlines is zero, resulting in \eqref{IntegrationStokesTrick}. We conclude that
		\begin{equation*}
			\area_g (\Sigma_{s_1+is_2}) \leq \area_g(\Sigma_{s_1+is_2}^t),
		\end{equation*}
		showing that the variation \eqref{VariationEq} for $F_t$ is nonnegative. Thus,
		\begin{equation}\label{PositiveVariation1}
			\begin{aligned}
				0 \leq&  \int_{\Sigma\cap \Omega_L} -|\gen{s(\cdot,\cdot),F_t}|^2 + |\nabla^\perp_\Sigma F_t|^2 - \tr_\Sigma \gen{\Rm_{X}(F_t,\cdot)\cdot, F_t} \, \dA_g \\ &+  \int_{\partial\Sigma} \gen{\nabla_{F_t} F_t, \nu} d A_g\\
				=& \int_{\Sigma \cap \Omega_L} - \chi_L^2|\gen{s(\cdot,\cdot),\hat x_1}|^2 + \left|\nabla_\Sigma \left(\chi_L \hat x_1 \right)\right|^2  - \chi_L^2 \tr_\Sigma \gen{\Rm_{X}(\hat x_1,\cdot)\cdot, \hat x_1} \, \dA_g\\
				&+ \int_{\partial\Sigma} \gen{\nabla_{\left(\frac{\hat x_1}{|\nabla x_1|}\right)} \left(\frac{\hat x_1}{|\nabla x_1|}\right), \nu}d A_g.
			\end{aligned}
		\end{equation}
		We can repeat the same arguments for $F_t = \chi_{L} \hat x_2$ to obtain
		\begin{equation}\label{PositiveVariation2}
			\begin{aligned}
				0 \leq& \int_{\Sigma \cap \Omega_L} - \chi_L^2|\gen{s(\cdot,\cdot),\hat x_2}|^2 + \left|\nabla_\Sigma \left(\chi_L \hat x_2 \right)\right|^2  - \chi_L^2 \tr_\Sigma \gen{\Rm_{X}(\hat x_2,\cdot)\cdot, \hat x_2} \, \dA_g\\
				&- \int_{\partial\Sigma} \gen{\nabla_{\left(\frac{\hat x_2}{|\nabla x_1|}\right)} \left(\frac{\hat x_2}{|\nabla x_1|}\right), \nu}d A_g.
			\end{aligned}
		\end{equation}
		Integrate the nonnegative \eqref{PositiveVariation1} and \eqref{PositiveVariation2} over $s\in \C$ and subtract them from \eqref{IntegralLeftover} to obtain
		\begin{equation}\label{VariationalInequality}
			\begin{aligned}
				&\int_{s\in\C}\int_{\Omega_L \cap \Sigma_s} - |\sff|^2 - \tr_\Sigma (\gen{\Rm_X(\hat x_1,\cdot)\cdot,\hat x_1} + \gen{\Rm_X(\hat x_2,\cdot)\cdot,\hat x_2}) \, \dA_g ds_1ds_2 \\
				\geq& \int_{s\in\C}\int_{\Omega_L \cap \Sigma_s} - (1-\chi_{L}^2) |\sff|^2 -(1-\chi_{L}^2) \tr_\Sigma\Big(\gen{\Rm_{X}(\hat x_1,\cdot)\cdot, \hat x_1} \\ &+ \gen{\Rm_{X}(\hat x_2,\cdot)\cdot, \hat x_2} \Big)
				- \chi_{L}^2\left(|\nabla^\perp_\Sigma \hat x_1|^2 + |\nabla^\perp_\Sigma \hat x_1|^2\right) - 2|\nabla_\Sigma \chi_{L}|^2   \, \dA_g ds_1ds_2 \\
				&- \int_{\partial\Sigma} \gen{\nabla_{\left(\frac{\hat x_1}{|\nabla x_1|}\right)} \left(\frac{\hat x_1}{|\nabla x_1|}\right), \nu} +  \gen{\nabla_{\left(\frac{\hat x_2}{|\nabla x_1|}\right)} \left(\frac{\hat x_2}{|\nabla x_1|}\right), \nu}d A_g.
			\end{aligned}
		\end{equation}
		
		We analyze these terms individually. First, for $\frac{L}{4}< u< \frac{L}{2}$:
		\begin{equation}\label{CutOffPermutation}
			\begin{aligned}
				|\nabla_\Sigma \chi_{L}|^2 &= \left|\frac{4}{L}\left(\nabla_\Sigma \frac{1}{ |\nabla x_1|}\right)u +  \frac{4}{L}\left(\frac{1}{ |\nabla x_1|}-1\right)\nabla_\Sigma u\right |^2 = \OO(L^{-2\tau-2})
			\end{aligned}
		\end{equation}
		by the AE asymptotics. Furthermore, $|\nabla_\Sigma \chi_L|^2 = |\nabla_\Sigma \frac{\hat x_1}{|\nabla x_1|}|^2 = \OO(r^{-2\tau-2})$ for $r\geq \frac{L}{2}$ and $\nabla_\Sigma \chi_L = 0$ for $r \leq \frac{L}{4}$. As $\tau > m-1$, the integral of $|\nabla_\Sigma \chi_{L}|^2$ vanishes as $L\to\infty$. Next,
		\begin{equation*}
			1-\chi_L^2 = \OO(r^{-\tau})
		\end{equation*}
		with $1-\chi_L^2 \to 0, L\to \infty$. As $\Rm_{X} = \OO(r^{-\tau-2})$, $|\sff| = \OO(r^{-\tau - 2})$, and $(1-\chi_L) = 0$ for $|u|\leq L$, dominated convergence implies that
		\begin{align*}
			&\int_{s\in\C}\int_{\Omega_L \cap \Sigma_s} - (1-\chi_{L}^2) \left(|\sff|^2 - \tr_\Sigma\left(\gen{\Rm_{X}(\hat x_1,\cdot)\cdot, \hat x_1} + \gen{\Rm_{X}(\hat x_2,\cdot)\cdot, \hat x_2} \right)\right)\\ &= \OO(L^{-2\tau-2+2m})
		\end{align*} 
		for some $\epsilon>0$. 
		
		For the final term in the integral over $\Sigma$, the complex coarea formula implies that
		\begin{align*}
			&\int_{s\in\C}\int_{\Omega_L \cap \Sigma_s} \chi_{L}^2(|\nabla^\perp_\Sigma \hat x_1|^2 + |\nabla^\perp_\Sigma \hat x_2|^2)  \, \dA_g ds_1ds_2\\
			=& \int_{\Omega_L} \chi_{L}^2 |\nabla x_1|^2 (|\nabla^\perp_\Sigma \hat x_1|^2 + |\nabla^\perp_\Sigma \hat x_2|^2) \dvol_g.
		\end{align*}
		As $J \hat x_1 = - \hat x_2$ and for a basis $e_3,\dots,e_{2m}\in T_p\Sigma$, $p\in \Sigma$, then
		\begin{equation}\label{HessianIneq}
			\begin{aligned}
				\chi_L^2 |\nabla x_1|^2 |\nabla^\perp_\Sigma \hat x_1|^2 &\leq \chi_L^2\sum_{k=3}^{2m}|\nabla x_1|^2 \gen{\nabla_{e_k} \hat x_1, \hat x_1}^2 + |\nabla x_1|^2 \sum_{k=3}^{2m} \gen{\nabla_{e_k} \hat x_1, \hat x_2}^2 \\
				&\leq \chi_L^2\frac{1}{2}|\nabla^2 x_2|^2.
			\end{aligned}
		\end{equation}
		as $\hat x_1$ and $\hat x_2$ are perpendicular and by compatibility with $J$. The same computation shows that 
		\begin{equation*}
			\chi_L^2 |\nabla x_1|^2 |\nabla^\perp_\Sigma \hat x_2|^2 \leq \chi_L^2\frac{1}{2}|\nabla^2 x_2|^2.
		\end{equation*}

		Finally, for the boundary term and $i=1$:
		\begin{align*}
			\gen{\nabla_{\left(\frac{\hat x_1}{|\nabla x_1|}\right)}\frac{\hat x_1}{|\nabla x_1|}, \nu } 
			&= \frac{1}{|\nabla x_1|^3} \gen{\nabla_{\hat x_1} \nabla x_1, \nu }.
		\end{align*}	
		As $\nu = \sum_{k=3}^{2m} \frac{x_k}{L} \del_k + \OO(r^{-\tau})$ and letting repeated indices denote summation over $1,\dots,2m$, then
		\begin{equation}\label{HessianExpansion}
			\begin{aligned}
				\gen{\nabla_{\hat x_1} \nabla x_1, \nu } &= \gen{\nabla_{\hat x_1} g^{k1}\del_k, \nu} \\
				&= g^{k1} \gen{\Gamma_{1k}^l \del_l, \nu} + g^{k1}_{\;\;\;,1} \gen{\del_k,\nu} \\
				&= g^{k1} \Gamma_{1k}^l \sum_{j=3}^{2m} g_{lj}\frac{x_j}{L} + g^{k1}_{\;\;\;,1} \sum_{j=3}^{2m} g_{kj} \frac{x_j}{L} + \OO(r^{-2\tau-1})\\
				&= -\frac{1}{2L} \sum_{j=3}^{2m}  g_{11,j} x_j + \OO(r^{-2\tau-1}).
			\end{aligned}
		\end{equation}
		By \eqref{HessianIneq} and \eqref{HessianExpansion}, using the coarea formula for the boundary term and subtracting the end result from \eqref{IntegralLeftover}, then
		\begin{equation}\label{key}
			\begin{aligned}
				&\int_{s\in\C}\int_{\Omega_L \cap \Sigma_s} \Big(- |\sff|^2 - \sum_{k=3}^{2m} (\gen{\Rm_X(\hat x_1,e_k)e_k,\hat x_1}\\ &+ \gen{\Rm_X(\hat x_2,e_k)e_k,\hat x_2}) \Big)\, \dA_g ds_1ds_2 \\
				\geq&  \int_{s\in\C} \Bigg(\int_{\Omega_L \cap \Sigma_s} \frac{1}{|\nabla x_1|^2}\left(- \frac{1}{2}|\nabla^2 x_1|^2 - \frac{1}{2}|\nabla^2 x_2|^2\right) \\
				&+ \frac{1}{2L} \int_{\partial (\Omega_L\cap \Sigma_s)} \frac{1}{|\nabla x_1|^3} \sum_{k=3}^{2m}  g_{11,k} x_k \dA_g \Bigg)ds_1ds_2\ + \OO(L^{-2\tau-2+2m})\\
				=& -\int_{\Omega_L} \frac{1}{2}|\nabla^2 x_1|^2 + \frac{1}{2}|\nabla^2 x_2|^2 \dvol_{g} + \frac{1}{2L}\int_{T_L} \frac{1}{|\nabla x_1|} \sum_{i=1}^2\sum_{k=3}^{2m}  g_{ii,k} x_k \dA_{\eucl}\\ &+ \OO(L^{-2\tau-2+2m}).
			\end{aligned}
		\end{equation}
		We could remove $\chi_L^2$ from $|\nabla^2 x_i|^2$ for $i=1,2$ up to an error $\OO(L^{-2\tau-2+2m})$ as $|\nabla^2 x_i|^2 = \OO(r^{-2\tau-2})$ and $\chi_L \equiv 1$ for $u\leq \frac{L}{4}$. As $\frac{1}{|\nabla x_1|} = 1 + \OO(r^{-\tau})$ and $g_{ii,k} = \OO(r^{-\tau-1})$, then
		\begin{align*}
			&\frac{1}{2L}\int_{T_L} \frac{1}{|\nabla x_1|} \sum_{i=1}^2\sum_{k=3}^{2m}  g_{ii,k} x_k \dA_{\eucl}\\ =& \frac{1}{2L}\int_{T_L} \sum_{i=1}^2\sum_{k=3}^{2m}  g_{ii,k} x_k \dA_{\eucl} + \OO(L^{-2\tau-2+2m}),
		\end{align*}
		proving the lemma.
	\end{proof}

	As the choice of $z=z_1$ above was arbitrary, the computations work for all $z_i$'s, and the lower bound of the ADM mass (Theorem \ref{MTheoremMassBound}) can now be proven.

	\begin{theorem}\label{TheoremMassBound}
	Let $(X^{2m},g,J)$ be an AE Kähler manifold with nonnegative and integrable scalar curvature $R_g$. Let $z = x_1 + ix_2$ be one of the holomorphic coordinate functions from Theorem \ref{BiholomorphismAsympt}. Then
	\begin{equation}\label{MassInequality}
	\m(g) \geq \frac{(m-1)!}{4(2m-1)\pi^m} \int_X \frac{1}{2}|\nabla^2 x_1|^2 + \frac{1}{2}|\nabla^2 x_2|^2 + \morm{\nabla x_1}^2 R_g \dvol_g.
	\end{equation}
	\end{theorem}

	\begin{proof}
	The proof follows by combining Lemma \ref{Lemma3}, \ref{MassScalarCurvIneqRightHand}, and \ref{LemmaCurvatureRestTerm} with \eqref{IntegralLemma2} for the domain $\Omega = \Omega_L$ and taking $L\to\infty$. This gives
	\begin{align*}
	&\lim_{L\to\infty} \Bigg( \int_{D_{1,L}^+} \sum_{k=1}^{2m}\left( g_{1k,k}  - g_{kk,1}\right) \dA_{\eucl}  +\int_{D_{2,L}^+}  \sum_{k=1}^{2m}\left( g_{2k,k}  - g_{kk,2}\right) \dA_{\eucl} 
	\\ &+\int_{D_{1,L}^-}  \sum_{k=1}^{2m}\left(- g_{1k,k}  + g_{kk,1}\right) \dA_{\eucl}  + \int_{D_{2,L}^-} \sum_{k=1 }^{2m}\left(- g_{2k,k}  + g_{kk,2}\right) \dA_{\eucl}\\
	&+\frac{1}{L} \int_{T_L} \sum_{i=1}^2  \sum_{k=3}^{2m}(g_{ik,i} -  g_{ii,k}) x_k \dA_{\eucl} +  \int_{T_L} \sum_{k,l=3}^{2m} (g_{kl,k} - g_{kk,l})x_l \dA_{\eucl} \Bigg) \\
	\geq& \lim_{L\to\infty} \int_{\Omega_L} \frac{1}{2}|\nabla^2 x_1|^2 + \frac{1}{2}|\nabla x_2|^2 + |\nabla x_1|^2 R_g \dvol_g.
	\end{align*}
	We recognize the left-hand side as the ADM mass $\m(g)$ up to the constant in \eqref{MassFormula}, proving the theorem.
	\end{proof}
	
	\begin{remark}
	There is a more direct way of proving Theorem \ref{TheoremMassBound} if one is willing to assume a higher-order decay on the metric and scalar curvature. Work by Qi Yao \cite[Theorem D]{yaoMass2022} shows that if for some large $k>0$:
	\begin{equation*}
		|\nabla^j(g-g_{\eucl})| = \OO(r^{-\tau-j}), \quad |\nabla^j R_g| = \OO(r^{-2\tau-2 - j}), \quad j=0,\dots,k,
	\end{equation*}
	then, for some dimensional constant $C(m)>0$, the Kähler form $\omega$ has the expansion:
	\begin{equation*}
		\omega =\begin{cases}
			dd^c r^{2} + C(m)\m(\omega)dd^cr^{4-2m} + \OO(r^{-2\tau})& m\geq 3\\
			dd^c r^{2} + \frac{3}{2}\m(\omega)dd^c\log(r) + \OO(r^{-2\tau})& m= 2.\\
		\end{cases}
	\end{equation*}
	Lemmas \ref{Lemma3}, \ref{MassScalarCurvIneqRightHand}, and \ref{LemmaCurvatureRestTerm} then directly equal some constants times the mass as $L\to\infty$. Especially Lemma \ref{Lemma3} is simplified as the right-hand side of \eqref{BoundaryLaplace1} then equals the mass times some constant, thus avoiding expanding $g_{11,1}$ and using the divergence theorem.\\
	
	Independent of this and in $\dim_{\R} X = 4$, Lemma \ref{LemmaCurvatureRestTerm} can also be proven by choosing the log-cutoff.
		\begin{equation*}
			\chi_{K,L,M} \coloneqq \mathbbm{1}_{|s| \leq M} \begin{cases}
				1 & |z_2| < K\\
				1- \frac{\log(\frac{|z_2|}{K})}{\log (\frac{L}{K})} & K \leq |z_2| \leq L \\
				0 & |z_2| > L,
			\end{cases}
		\end{equation*}
		and first taking $L\to\infty, K\to\infty$, and finally $M\to\infty$. This cutoff was extensively used in the proof of the Positive Mass Theorem by Schoen-Yau \cite{schoenproof1979}. For other applications, see \cite{coldingCourse2011}.
	\end{remark}

	From Theorem \ref{TheoremMassBound}, we give a new proof of the Positive Mass Theorem for AE Kähler manifolds.
	
	\begin{corollary}\label{CorollaryPMT}
		\sloppy Given $(X^{2m},g,J)$ as in Theorem \ref{TheoremMassBound}, the $\m(g)\geq 0$ and $\m(g)=0$ if and only if $(X,g,J)$ is isometric and biholomorphic to $(\C^m, g_{\eucl}, J_{\eucl})$.
	\end{corollary}
	
	\begin{proof}
	The ADM mass is nonnegative by \eqref{MassInequality}, as the scalar curvature is nonnegative. For the equality case, take $\m(g)=0$ with nonnegative scalar curvature. Then, we have:
	\begin{equation*}
		-\gen{\clubsuit c_1(X),[\omega]} = 0,
	\end{equation*}
	implying that $X$ is biholomorphic to $\C^m$. Furthermore, $\nabla^2 x_i = 0$ for all $i=1,\dots,2m$, and since the $\nabla x_i$'s are linearly independent for all $i$, we obtain a basis of parallel vector fields $(\nabla x_1,\dots, \nabla x_{2m})$ at all points. The proof of the fact that a manifold is locally isometric to Euclidean spaces if and only if it is flat (see \cite[Theorem 7.10]{leeIntroduction2018}) shows that $(X,g)$ is a complete flat Riemannian manifold and hence isometric to $g_{\eucl}$.
	\end{proof}

	\section{Stability of the Positive Mass Theorem}

Given the ADM mass inequality of Theorem \ref{TheoremMassBound}, we are ready to prove two new stability results of the Positive Mass Theorem on AE Kähler manifolds. The first result holds for general sequences of AE Kähler manifolds with ADM masses converging to zero. In this case, it is necessary to remove certain sets with controlled and vanishing areas in the limit. The following is Theorem \ref{MCorollaryStabNoRic} of the introduction.

\begin{theorem}\label{CorollaryStabNoRic}
	Let $(X_i^{2m},g_i, J_i)$ be a sequence of AE Kähler manifolds of fixed complex dimension $m$ with nonnegative and integrable scalar curvature, and suppose that the ADM masses $\m(g_i) \to 0$ as $i\to\infty$. Then for all $i$, there exists a domain $Z_i\subset X_i$ such that $X_i\setminus Z_i$ converges in the pointed Gromov-Hausdorff sense to $(\R^{2m},g_{\eucl})$:
	\begin{equation*}
		(X_i\setminus Z_i,\hat d_{g_i},p_i) \, \overset{pGH}{\longrightarrow} \, (\R^{2m}, d_{\eucl},0),
	\end{equation*}
	where $p_i\in M_i\setminus Z_i$ is any choice of basepoint and $\hat d_{g_i}$ is the induced length metric of $g_i$ on $X_i\setminus Z_i$. Furthermore, for any continuous $\xi\colon (0,\R) \to (0,\infty)$ such that $\lim_{x\to 0^+} \xi(x) = 0$, we have for $i$ large enough:
	\begin{equation*}
		\area(\partial Z_i) \leq \frac{\m(g_i)^{\frac{m}{2m-2}+\frac{1}{2}}}{\xi(\m(g_i))}.
	\end{equation*}
\end{theorem}
	
\begin{proof}
	Theorem \ref{CorollaryStabNoRic} was proved by Dong-Song on three-dimensional Riemannian manifolds using the integral inequality \cite{brayHarmonic2022}, the (real) coarea formula, and basic estimates. We will now show how to adapt the proof to AE Kähler manifolds of any dimension. Let $(z_1,\dots, z_m) = (x_1+ix_2,\dots, x_{2m-1}+i x_{2m})$ be the holomorphic coordinates obtained from Theorem \ref{BiholomorphismAsympt}. On $(X,g)$, we define the smooth function
	\begin{equation*}
		F(p) = \sum_{j,k=1}^{2m} \left(g(\nabla x_j,\nabla x_k) - \delta_{jk} \right)^2.
	\end{equation*}
	The proof consists of two parts.\\
	
	\textbf{Part 1} (see \cite[Section 3]{dongStability2023}): For the first part, let $\Omega_{a,b} \coloneqq F^{-1}([a,b])$. Choose $\epsilon\ll 1$ small enough to obtain
	\begin{equation}\label{Eq2DongSong}
		|\text{Jac} (x)- \text{Id}|_g\leq \epsilon'
	\end{equation}
	on the set $\Omega_{0,6\epsilon}$ for some other small $\epsilon' \coloneqq 100\sqrt{\epsilon}\ll 1$. Next, on the same set, we have
	\begin{equation}\label{Eq4DongSong}
		\begin{aligned}
			|\nabla F(p)| &= \sum_{j,k=1}^{2m}2|g(\nabla x_j, \nabla x_k) - \delta_{jk}|\cdot(|\nabla x_j||\nabla^2 x_k| + |\nabla^2 x_j||\nabla x_k| )\\ &\leq C \sum_{j=1}^{2m} |\nabla^2 x_j|.
		\end{aligned}
	\end{equation}
	Now comes the only time that \cite{dongStability2023} uses the ADM mass inequality \cite[Theorem 1.2]{brayHarmonic2022}: Combining \eqref{Eq2DongSong} and \eqref{Eq4DongSong}, we have
	\begin{equation}\label{MassFIneq}
		\begin{aligned}
			\int_{\Omega_{0,6\epsilon}} |\nabla F(p)|^2 &\leq C \sum_{j=1}^{2m}\int_{\Omega_{0,6\epsilon}}|\nabla^2 x_j|^2\\
			&\leq C \sum_{j=1}^{2m}\int_{\Omega_{0,6\epsilon}}\frac{|\nabla^2 x_j|^2}{|\nabla x_j|}\\
			&\leq C \m(g).
		\end{aligned}
	\end{equation}
	In their case, they replace $2m$ by $3$ in the sum. The ADM mass inequality of \cite{brayHarmonic2022} involves integrating $\frac{|\nabla^2 x_j|^2}{|\nabla x_j|}$ and thus gives the last inequality. However, as the inequality \eqref{MassInequality} only has the integrand $|\nabla^2 x_j|^2$ (plus the term with the scalar curvature), we can skip the penultimate inequality of \eqref{MassFIneq} and conclude that
	\begin{equation*}
		\int_{\Omega_{0,6\epsilon}} |\nabla F|^2 \leq C \m(g)
	\end{equation*}
	in our setting of AE Kähler manifolds.\\
		
		From this, they use the coarea formula and the isoperimetric inequality to find a $\tau_0 \in [0,6\epsilon]$ such that
		\begin{equation*}
			\area (S_{\tau_0}) \leq \frac{C\m(g)^2}{\epsilon^4},
		\end{equation*}
		where $S_{\tau_0} = F^{-1}(\tau_0)$. Mutatis mutandis, the exponents change due to dimensional reasons and become
		\begin{equation}\label{AreaSBoundedEpsilon}
			\area (S_{\tau_0}) \leq \frac{C\m(g)^{\frac{m}{2m-2}+\frac{1}{2}}}{\epsilon^{\frac{2m}{2m-2}+1}}.
		\end{equation}
		Via $S_{\tau_0}$ and continuing to follow \cite{dongStability2023}, we can find a set $E\subset X$ connected to infinity such that $\Psi|_E$, where $\Psi$ is the map $\colon X \to \C^m$ from Theorem \ref{TheoremMassBound}, is a diffeomorphism and $\area (\partial \Psi(E))$ is small.
		
		\textbf{Part 2} (see \cite[Section 4]{dongStability2023}): In the next and final section of \cite{dongStability2023}, Dong and Song apply the isoperimetric inequality and coarea formula to show that $\Psi(E)$ (after slightly modifying it) converges to $\R^{2m}$ in the pointed Gromov-Hausdorff sense, and $\area_{\eucl}(\partial \Psi(E)) \to 0$ as $\m(g) \to 0$. This analysis is almost independent of the dimension and only requires an estimate of the form
		\begin{equation}\label{AreaSBounded}
			\area (S_{\tau_0}) \leq f(\m(g))
		\end{equation}
		for a continuous function $f\colon \R_+ \to \R_+$ on the right-hand side, such that $\lim_{y\to 0^+} f(y) = 0$.
		
		We first introduce the necessary notation from \cite{dongStability2023}, generalized to our setting, and starting from \eqref{15} and ending at Remark \ref{Remark4.2}, construct the extra steps of the proof required in our setting. Any arguments coming after these steps to finish the proof are straightforward generalizations of proofs in \cite{dongStability2023}. One should always think of the mass $\m(g)$ as being very small.\\
	 	
	 	Fix a continuous function $\xi \colon (0,\infty)\to (0,\infty)$ such that
	 	\begin{equation*}
	 		\lim_{y\to 0^+} \xi(y) = 0, \quad \lim_{y\to 0^+} \frac{y}{\xi(y)} = 0.
	 	\end{equation*}
	 	Choose continuous functions $\xi_0,\xi_1 \colon (0,\infty) \to (0,\infty)$ with $\lim_{y\to 0^+} \xi_0(y) = \lim_{y\to 0^+} \xi_1(y) = 0$, and
	 	\begin{equation*}
	 		\lim_{y\to 0^+} \frac{\xi(y)}{\xi_0^{100m}(y)} = \lim_{y\to 0^+} \frac{\xi_0(y)}{\xi_1^{100m}(y)} = 0.
	 	\end{equation*}
	 	One should think of these functions converging very slowly to $0$ as $y\to 0^+$. Finally, set
	 	\begin{equation*}
	 		\delta_0 \coloneqq \xi_0(\m(g)), \quad \delta_1 \coloneqq \xi_1(\m(g)),
	 	\end{equation*}
	 	so that
	 	\begin{equation*}
	 		\delta_1^{100m} \gg \delta_0 \gg \xi(\m(g))^{\frac{1}{100m}} \gg \m(g)^{\frac{1}{100m}},
	 	\end{equation*}
	 	and set $\epsilon = \delta_0$ in \eqref{AreaSBoundedEpsilon}.
	 	
	 	Let $E$ be the bounded region with boundary $S_{\tau_0}$. Push $S_{\tau_0}$ forward to $\C^{m}\cong\R^{2m}$ as $\Psi(S_{\tau_0})\eqqcolon \Sigma$ via $\Psi$ from Theorem \ref{BiholomorphismAsympt}, and let $\W$ be the compact domain with boundary $\Sigma$. From now on, the function $f$ of \eqref{AreaSBounded} may change from line to line but always has the property that $\lim_{y\to 0^+}f(y) = 0$. Given an $n$-dimensional hypersurface $D^n$ for $0\leq n \leq 2m$, denote by $\area_{n}(D)$ the $n$-dimensional Euclidean Hausdorff measure of $D$. This means that $\vol(\cdot) = \area_{2m}(\cdot)$ and $\length(\cdot) = \area_1(\cdot)$. By the isoperimetric inequality and \eqref{AreaSBoundedEpsilon}, then
	 	\begin{equation*}
	 		\vol(\W) \leq C \area_{2m-1}(\Sigma)^{\frac{2m}{2m-1}} \leq f(\m(g)).
	 	\end{equation*}
	 	For any $\mathbf{k} = (k_1,\dots, k_{2m}) \in \Z^{2m}$, consider the cube $C_{\k}^{2m}(\delta_1)$ given by
	 	\begin{equation*}
	 		C_{\k}^{2m}(\delta_1) \coloneqq (k_1\delta_1,(k_1+1)\delta_1) \times \cdots \times (k_{2m}\delta_1,(k_{2m}+1)\delta_1) \subset \R^{2m}.
	 	\end{equation*}
	 	Let $B_{\k}(r)$ be the Euclidean ball with the same center as $C_{\k}^{2m}(\delta_1)$ and radius $r$. By applying the coarea formula, we can find $r\in (3\delta_1, 3\delta_1 + \delta_0)$ such that $\W\cap \partial B_{\k}(r)$ consists of smooth $(2m-1)$-dimensional manifolds and 
	 	\begin{equation}\label{14}
	 		\area_{2m-1}(\W \cap \partial B_{\k}(r)) \leq \frac{\vol(\W\cap B_{\k}(4\delta_1))}{\delta_0}.
	 	\end{equation}
	 	The relative isoperimetric inequality \cite[Theorem 5.11(ii)]{gariepyMeasure2015} shows that
	 	\begin{align*}
	 		\vol(\W \cap B_{\k}(4\delta_1)) &\leq C \area_{2m-1}(\Sigma \cap B_{\k}(4\delta_1))^{\frac{2m}{2m-1}}\\
	 		&\leq f(\m(g)) \area_{2m-1}(\Sigma \cap B_{\k}(4\delta_1)).
	 	\end{align*}
	 	So by \eqref{14}:
	 	\begin{align*}
	 		\area_{2m-1}(\W \cap \partial B_{\k}(r)) &\leq f(\m(g)) \area_{2m-1}(\Sigma \cap B_{\k}(4\delta_1))\\
	 		&\leq \area_{2m-1}(\Sigma \cap B_{\k}(4\delta_1)).
	 	\end{align*}
	 	This allows us to smooth the hypersurface $(\Sigma \cap B_{\k}(r))\cup (\W \cap \partial B_{\k}(r))$ to get a closed embedded hypersurface $\Sigma_{\k}^{2m-1} \subset B_{\k}(4\delta_1)$ satisfying
	 	\begin{equation}\label{15}
	 		\begin{aligned}
	 			\area_{2m-1}(\Sigma_{\k}^{2m-1}) &\leq 2\area_{2m-1}(\Sigma \cap B_{\k}(4\delta_1)) + 2\area_{2m-1}(\W \cap \partial B_{\k}(r))\\
	 			&\leq 4 \area_{2m-1}(\Sigma \cap B_{\k}(4\delta_1)).
	 		\end{aligned}
	 	\end{equation}
	 	
	 	The main new point of our construction compared to Dong-Song \cite{dongStability2023} starts here, where in dimension three, only one step is needed, but $(2m-2)$-steps in real dimension $2m$: for $t\in \R$, define the $(2m-1)$-dimensional hyperplane
	 	\begin{equation*}
	 		A^{2m-1}_{\k,\delta_1}(t) \coloneqq \{ (x_1,\dots, x_{2m}) \mid x_{2m} = (k_{2m}+t)\delta_1\}.
	 	\end{equation*}
	 	By definition $C_{\k}^{2m}(\delta_1) \subset \cup_{t\in [0,1]} A_{\k,\delta_1}^{2m-1}(t)$. By the coarea formula there exists $t_{\k,2m-1} \in(\frac{1}{2},\frac{1}{2}+\delta_0)$ such that $A_{\k,\delta_1}^{2m-1}(t_{\k,2m-1}) \cap \Sigma_{\k}^{2m-1}$ consists of $(2m-2)$-dimensional hypersurfaces and
	 	\begin{equation}\label{16}
	 		\begin{aligned}
	 			\area_{2m-2}(A_{\k,\delta_1}^{2m-1}(t_{\k,2m-1}) \cap \Sigma_{\k}^{2m-1}) &\leq \frac{\area_{2m-1}(\Sigma_{\k}^{2m-1})}{\delta_0\delta_1}\\
	 			&\leq f(\m(g)).
	 		\end{aligned}
	 	\end{equation}
		
		\sloppy We repeat this process inductively. We find hyperplanes $A_{\k,\delta_1}^d(t_{\k,d}) = \{(x_1,\dots,x_{2m}) \mid x_{d+1} = t_{\k,d}, \dots, x_{2m} = t_{\k,2m-1} \}$ of real dimension $d$ with orthogonal projections $\pi_d\colon A_{\k,\delta_1}^d(t_{\k,d}) \to A_{\k,\delta_1}^{d-1}(t_{\k,d-1})$ such that the following is true: we define inductively $\Sigma_{\k}^d \coloneqq \Sigma^{d+1}_{\k} \cap A_{\k,\delta_1}^d(t_{\k,d})$, where $t_{\k,d+1} \in(\frac{1}{2},\frac{1}{2}+\delta_0)$ is chosen using the coarea formula such that
		\begin{equation}\label{16Ext}
			\begin{aligned}
				\area_{d}(\Sigma_{\k}^{d}) &\leq \frac{\area_{d+1}(\Sigma_{\k}^{d+1})}{\delta_0\delta_1}\\
				&\leq f(\m(g)).
			\end{aligned}
		\end{equation}
		Finally, define $C_{\k}^d(\delta_1) \coloneqq C_{\k}^{2m}(\delta_1) \cap A_{\k,\delta_1}^d(t_{\k,d})$.\\
		
		On $C_{\k}^2(\delta_1)$, let $C_{\k}^2(\delta_1)'$ be the connected component of $C_{\k}^2(\delta_1) \setminus \Sigma_{\k}^1$ with biggest volume and $E_{\k}^2$ the complement. The relative isoperimetric inequality and \eqref{16Ext} implies that
		\begin{equation}\label{17}
			\begin{aligned}
				\area_{2}(E_{\k}^{2}) &\leq C \area_{1}(A_{\k,\delta_1}^{2}(t_{\k,2}) \cap \Sigma_{\k}^{2})^{2}\\
				&\leq f(\m(g))\area_{2}(\Sigma_{\k}^{2})\\
				&\leq \area_{2}(\Sigma_{\k}^{2}).
			\end{aligned}
		\end{equation}
		From this, define $C_{\k}^3(\delta_1)'\coloneqq C_{\k}^2(\delta_1)' \cup \left(C_{\k}^3(\delta_1) \cap \pi_3^{-1}(C_{\k}^2(\delta_1)' \setminus \pi_3(\Sigma_{\k}^2)) \right)$.
		The following two facts are similar to \cite[Lemma 4.1 and 4.2]{dongStability2023}:
		\begin{enumerate}[(i)]
			\item $C_{\k}^{3}(\delta_1)'$ is path connected.
			\item $\area_{3}(C_{\k}^{3}(\delta_1) \setminus C_{\k}^{3}(\delta_1)') \leq f(\m(g))$. \label{Fact2}
		\end{enumerate}
		We only prove the second statement; the first one is almost by definition. Since 	
		\begin{equation}\label{AreaFormula}
			\area_{2}(\pi_{3}(\Sigma_{\k}^{2})) \leq 	\area_{2}(\Sigma_{\k}^{2}),
		\end{equation}
		and by \eqref{17}:
		\begin{equation*}
			\area_{2}(E_{\k}^{2}) \leq \area_{2}(\Sigma_{\k}^{2}),
		\end{equation*}
		then
		\begin{equation*}
			\area_{3}(C_{\k}^{3}(\delta_1) \setminus C_{\k}^{3}(\delta_1)')  \leq 2\delta_1\area_{2}(\Sigma_{\k}^{2}) \leq f(\m(g)).
		\end{equation*}
		We continue this process iteratively: define
		\begin{equation*}
			C_{\k}^d(\delta_1)' \coloneqq C_{\k}^{d-1}(\delta_1)' \cup \left(C_{\k}^d(\delta_1) \cap \pi_d^{-1}(C_{\k}^{d-1}(\delta_1)' \setminus \pi_d(\Sigma_{\k}^{d-1})) \right).
		\end{equation*}
		By \eqref{16Ext}, one shows \eqref{17} (with another exponent), and we obtain fact \eqref{Fact2} for $C_{\k}^{d}(\delta_1)'$ by induction. Hence, it follows that there exists $C_{\k}^{2m}(\delta_1)' \subset C_{\k}^{2m}(\delta_1)$ path connected such that
		\begin{equation*}
			\vol\left(C_{\k}^{2m}(\delta_1) \setminus C_{\k}^{2m}(\delta_1)'  \right) \leq  f(\m(g))
		\end{equation*}
		which is Lemma 4.2 of \cite{dongStability2023} in real dimension $2m$. Hence, $C_\k^{2m}(\delta_1)'$ almost fills out the entire cube $C_\k^{2m}(\delta_1)$ as $i\to\infty$, it is path-connected, and there is a projection $\pi\colon C_\k^{2m}(\delta_1)' \to A^2_{\k,\delta_1}(t_{\k,2})$ not going through $\W$.
		
		\begin{remark}\label{Remark4.2}
			It is not possible to use a higher-dimensional version of the coarea formula to directly find a $t_{\k,2}$ and projection $\pi\colon C^{2m}_{\k}(\delta_1)\to  A^2_{\k,\delta_1}(t_{\k,2})$ as there is no equivalent of the area formula \eqref{AreaFormula} because of $\dim(\Sigma_\k^{2m-1}) > \dim(A^2_{\k,\delta_1}(t_{\k,2}))$.
		\end{remark}
		
		Finally, \cite[Lemma 4.4]{dongStability2023} shows that $\diam(C_{\k}^{2m}(\delta_1)') \leq 5\delta_1$. The proof works exactly as in our case (but modifying $5\delta_1$), namely for any two points $x_1,x_2\in C_{\k}^{2m}(\delta_1)'$, consider the projections
		\begin{equation}\label{ProjectionsX}
			x_i' \coloneqq \pi_3\circ \cdots \circ \pi_{2m}(x_i) \in C_{\k}^{2}(\delta_1)', \quad i=1,2.
		\end{equation}
		Letting $d(\cdot,\cdot)$ denote the induced distance in $C_{\k}^{2m}(\delta_1)'$ via the Euclidean metric, then $d(x_i,x_i') \leq 4m\delta_1$ and
		\begin{equation*}
			d(x_1',x_2') \leq d_{\eucl}(x_1',x_2') + \area_{1}(\Sigma_{\k}^1) \leq 2\delta_1 + f(\m(g)),
		\end{equation*}
	finishing the proof via the triangle inequality. \\
	
	The arguments above show that we can control the Gromov-Hausdorff distance between points in $C_{\k}^{2m}(\delta_1)'$ due to the diameter bound, namely that $d(x,x')\leq 5\delta_1$ for all $x,x'\in C_{\k}^{2m}(\delta_1)'$. To further control the distance between points $x\in C_{\k}^{2m}(\delta_1)'$ and $y\in C_{\k'}^{2m}(\delta_1)'$ for $\k \neq \k'$, Dong-Song uses the coarea formula to show that there exist points $x'\in C_{\k}^{2m}(\delta_1)'$ and $y'\in C_{\k'}^{2m}(\delta_1)'$ such that the straight line between $x'$ and $y'$ does not intersect $\Sigma$. The triangle inequality thus controls the induced distance $d(x,y)$, and we obtain GH control away from $\W$ as desired. 
	\end{proof}

	\begin{remark}
		
		After the arXiv version of this paper appeared, Dong \cite[Appendix]{dongStability2024} also extended the arguments of \cite[Section 4]{dongStability2023} to higher dimensions. His approach to the key part of the proof (between \eqref{15} and Remark \ref{Remark4.2}) is different from ours. In fact, he proves a global theorem by induction on the dimension, rather than iterating the projection argument from \cite{dongStability2023} in each local cube.
	\end{remark}

	For other applications, Kazaras-Khuri-Lee \cite{kazarasStability2021} proved Gromov-Hausdorff convergence to $(\R^3,d_{\eucl})$ for a sequence of three-manifolds $(M^3_i,g_i)$ with a lower bound on the Ricci tensor (plus other conditions) and with ADM masses $(M^3_i,g_i)\to 0$. As the integral inequality of \cite{brayHarmonic2022} was an essential step combined with methods to prove the Cheeger-Colding Almost Splitting Theorem, we directly carry over their proof to the Kähler setting via Theorem \ref{TheoremMassBound} to prove the next theorem (Theorem \ref{MCorollaryStabRicciBound}).

	\begin{definition}\label{Defbtau}
	Given $b>0, \tau>1m-1$, an AE Kähler manifold $(X^{2m},g,J)$ is $(b,\tau)$-asymptotically flat if in the given AE coordinate chart:
	\begin{equation*}
	|\del^k(g_{ij}-\delta_{ij})| = br^{-\tau-k}, \quad k=0,1,2.
	\end{equation*}
	\end{definition}

	Any $(b,\tau)$-asymptotically flat Kähler manifold is then uniformly AE at infinity and with decay rate $-\tau$.
	
	\begin{theorem}\label{CorollaryStabRicciBound}
		Fix $b>0$, $\tau>m-1$, $\kappa>0$, and a point $p\in \R^{2m} \setminus \overline{B_1^{\eucl(0)}}$. Take a pointed sequence $(X_i^{2m}, g_i, J_i,p_i)$, $p_i \in X_i$, of $(b,\tau)$-asymptotically flat Kähler manifolds with nonnegative and integrable scalar curvatures and Ricci curvatures
		\begin{equation*}
			\Ric_{g_i} \geq -2\kappa g_i.
		\end{equation*} 
		Furthermore, fix the AE charts $(\Phi_i)$ of $(X_i, g_i)$ such that $\Phi_i(p_i) = p$. If $\m(g_i)\to 0$, then $(X_i,d_{g_i},p_i)$ converges to $(\R^{2m},d_{\eucl}, p)$ in the pointed Gromov-Hausdorff sense.
	\end{theorem}
	
	As \eqref{MassInequality} is an integral on $X$ and not on the unbounded component away from all closed minimal hypersurfaces as in \cite{brayHarmonic2022}, Theorem \ref{CorollaryStabRicciBound} does not require the condition $\H_2(X, \Z) = 0$ appearing in \cite[Theorem 1.2]{kazarasStability2021}.\\
	
	The major difference between Theorems \ref{CorollaryStabNoRic} and \ref{CorollaryStabRicciBound} is that the latter does not remove any sets for the pointed Gromov-Hausdorff convergence. However, Theorem \ref{CorollaryStabRicciBound} simultaneously requires a lower bound on the Ricci curvature and uniform asymptotics at infinity, severely restricting the types of singularities that could appear along the sequence.
	
	\begin{proof}
		On the AE three-manifold $(M^3,g)$, Section 3 of \cite{kazarasStability2021} uses the Cheng-Yau gradient estimate to obtain a bound on the harmonic coordinates $(u_1,u_2,u_3)$ at infinity to show that
		\begin{equation}\label{EqGradtoNoGrad}
			\int_{M^3} \frac{|\nabla^2 u_i|^2}{|\nabla u_i|} \dvol_g \leq C \m(M^3,g )\Rightarrow 	\int_{M^3} |\nabla^2 u_i|^2 \dvol_g \leq C \m(M^3,g),
		\end{equation}
		and the latter sections only use the inequality on the right-hand side. As the harmonic coordinates $(u_1,u_2,u_3)$ are replaced by the holomorphic coordinates $(z_1,\dots,z_m)$ coming from Theorem \ref{BiholomorphismAsympt} in the Kähler case, the ADM mass inequality of Theorem \ref{TheoremMassBound} is equivalent to the right-hand side of \eqref{EqGradtoNoGrad}, and this section is unnecessary in our case. The next sections use Bishop-Gromov volume comparison and Cheeger-Colding's Almost Splitting Theorem \cite{cheegerLower1996}, as well as the methods used to prove this theorem, to finish the proof. These methods only depend on a lower bound on the Ricci curvature and so carry over to any dimension.
	\end{proof}

	\section{Examples of Families of Kähler Metrics on $\C^2$}\label{SectionExamples}
	
	We construct three families of AE Kähler metrics on $(\C^2, J_0)$ with positive and integrable scalar curvatures and the standard complex structure $J_0$. These families are parameterized by $\lambda>0$ and denoted as $(\C^2, g_\lambda, J_0)$. As $\lambda\to 0$, $\m(g_\lambda)\to 0$. The results of Theorems \ref{TheoremMassBound} and \ref{CorollaryStabNoRic} thus apply, and we conclude that by cutting out suitable sets $Z_i\subset \C^2$, $(\C^2\setminus Z_i, d_{\hat g_i}, p) \to (\C^2\setminus, d_{g_{\eucl}}, 0)$ in the pointed GH sense. 
	
	Two of the families have Ricci curvatures unbounded from below, while one family has Ricci curvature bounded from below. Therefore, Theorem \ref{CorollaryStabRicciBound} can be applied to this family. All families converge to $g_{\eucl}$ away from $0$ in the $C^0$-sense as $\lambda\to 0$.
	
	Even if a singularity appears in the family, the order is not great enough to destroy the GH convergence. In other words, all families considered converge to $(\R^4, d_{\eucl})$ in the GH-sense, as shown by direct computations. Applying Theorems \ref{CorollaryStabNoRic} or \ref{CorollaryStabRicciBound} is hence superfluous, but the examples do show that the families of metrics in the theorems are nonempty. Examples of families developing a gravitational well would be interesting for the application of Theorem \ref{CorollaryStabNoRic}.\\

	\subsection{Smoothings of a $\log$-singularity with Ricci curvature unbounded from below}\label{SectionLogSing}
	
	The first example is a family of smoothings of the scalar-flat Kähler metric, usually referred to as the Burns metric:
	\begin{equation*}
		\omega^{\log} = i\del\delbar(r^2 + \log r^2).
	\end{equation*}
	$\omega^{\log}$ has a singularity at $0\in \C^2$, but becomes a smooth metric on the blow-up of $0\in \C^2$ \cite[p. 594]{lebrunCounterexamples1988}. For the first example, let $(z_1,z_2)\in \C^2$ be the usual holomorphic coordinates, $\lambda >0$, and consider the smoothing
	\begin{equation}\label{OmegaLog}
		\omega^{\log}_\lambda = i\del\delbar (r^2 + \lambda\log (r^2 + \lambda)) = \begin{pmatrix}
			\frac{2\lambda^2 + \lambda(2r^2 + |z_2|^2) + r^4}{(r^2 + \lambda)^2} & - \frac{\lambda \overline{z_1}z_2}{(r^2 + \lambda)^2} \\
			- \frac{\lambda z_1\overline{z_2}}{(r^2 + \lambda)^2} & 	\frac{2\lambda^2 + \lambda(2r^2 + |z_1|^2) + r^4}{(r^2 + \lambda)^2}
		\end{pmatrix}.
	\end{equation}
	As $\del\delbar \log (r^2 + \lambda) = \OO(r^{-2})$ the family $(\omega^{\log}_\lambda)$ is AE and $\del\delbar$-exact. Finally, as all $\omega_\lambda^{\log}$ are $\text{U}(n)$-symmetric and by restricting to $\{z_2 =0\}$, the eigenvalues appear on the diagonal and are positive, so we conclude that all $\omega_\lambda^{\log}$ are positive definite.\\
	
	The diagonalized Ricci curvature (or equivalently at $z_2 = 0$) is
	\begin{align*}
	&\Ric_{\omega_\lambda^{\log}}= -\del\delbar \log \left(\omega_\lambda^{\log}\right)^2 \\
	&\diag\big(\lambda \frac{-|z_1|^{12}-	\lambda |z_1|^{10} -28\lambda^2|z_1|^8-40\lambda^3|z_1|^6-8\lambda^4|z_1|^4+32\lambda^5|z_1|^2+24\lambda^6}{(|z_1|^2+\lambda)^2(|z_1|^2+2\lambda)^2(|z_1|^4+2\lambda|z_1|^2+2\lambda^2)^2}), \\ 
	&\frac{\lambda |z_1|^4+4\lambda^2|z_1|^2+6\lambda^3}{|z_1|^8 + 5\lambda |z_1|^6+10\lambda^2|z_1|^4+10\lambda^3|z_1|^2+4\lambda^4}\big)
	\end{align*}
	and so the scalar curvature is (by spherical symmetry)
	\begin{equation*}
		R_{\omega^{\log}_\lambda} = \frac{2\lambda^3r^6 + 20\lambda^4r^4+35\lambda^5r^2 + 24\lambda^6}{(r^2+2\lambda)(r^4+ 2\lambda r^2+ 2\lambda^2)^3} = \OO(r^{-8}).
	\end{equation*}
	One sees that, at $r = \sqrt \lambda$, the first eigenvalue of the Ricci curvature equals
	\begin{equation*}
		-\frac{0.076}{\lambda}
	\end{equation*}
	after rounding, so the Ricci curvature is unbounded from below as $\lambda \to 0$. Similar analysis shows that the second eigenvalue is not bounded from above. Finally, the decay condition on the scalar curvature implies that $R_{\omega^{\log}_\lambda}\in L^1_{\omega^{\log}_\lambda}(\C^2)$. \\
	
	As $c_1(\C^2) = 0$, the mass $\m( \omega^{\log}_\lambda)$ can be computed via Theorem \ref{MassTheorem} as the integral of the scalar curvature:
	\begin{equation*}
		R_g \cdot (\omega_\lambda^{\log})^2 = \frac{8(2\lambda^3r^6 + 20\lambda^4r^4 + 36\lambda^5r^2+24\lambda^6)}{(r^2+2\lambda)(r^4+2\lambda r^2 + 2\lambda^2)^3} \dvol_{\eucl}.
	\end{equation*}
	To estimate the integral of this form, we only keep the highest and lowest order of the polynomial in the numerator and denominator. Hence, we obtain
	\begin{equation*}
		\m( \omega^{\log}_\lambda) = \int_{\C^2} R_{\omega^{\log}_\lambda} ({\omega^{\log}_\lambda})^2 \propto  \int_{r=0}^\infty \frac{\lambda^3 r^6 + \lambda^6}{r^{14} + \lambda^7} r^3 dr = \OO(\lambda\log(\lambda))
	\end{equation*}
	as $\lambda\to 0$. This computation was checked using a computer. Hence, Theorem  \ref{CorollaryStabNoRic} applies to $(\omega^{\log}_\lambda)$ as it is also sufficiently regular at infinity.\\

	Finally, we can compute the distance between $(0,0)\in \C^2$ and any point $p \in \del B_1^{\eucl}(0)$. By spherical symmetry, take $z_1 = (1,0)$. Take the curve $\gamma \colon [0,1] \to \C^2, t\mapsto (t,0)$, which again by spherical symmetry is the shortest path between $(0,0)$ and $(1,0)$:
	\begin{equation}
		\begin{aligned}
			\length_{\omega}(\gamma)&= \int_0^1|\gamma'(t)|_{\omega^{\log}_\lambda} dt\\
			&= \int_0^1 \sqrt{(\omega^{\log}_\lambda)_{11 }} dt
		\end{aligned}
	\end{equation}
	In this case:
	\begin{align*}
		\length_{\omega^{\log}_\lambda}(\gamma)&= \int_0^1|\gamma'(t)|_{\omega^{\log}_\lambda} dt\\
		&= \int_0^1 \sqrt{(\omega^{\log}_\lambda)_{11}} dt\\ 
		&= \int_0^1  \sqrt{\frac{|z_1|^4+2\lambda|z_1|^2+2\lambda^2}{(|z_1|^2+\lambda)^2}} dt \\
		&= \int_0^1  \sqrt{1 + \frac{\lambda^2}{(|z_1|^2+\lambda)^2}} dt \to 1, \quad \lambda \to 0.
	\end{align*}
	This means that there is no gravitational well at $0\in \C^2$ for $\omega^{\log}_\lambda$ as $\lambda \to 0$, and we obtain global GH convergence to $(\R^4, d_{\eucl})$ with the induced length metric $d_{\eucl}$. By changing the factor of $\lambda$ multiplied by $\log(r^2+\lambda)$ to $h(\lambda)$, where $h(\lambda)\to 0$ as $\lambda\to 0$, we ensure that $h(\lambda)\leq \frac{3}{2}\lambda$ to preserve the nonnegative scalar curvature. Since the Burns metric is smooth when blowing up $0$, we can consider $h(\lambda)$ as correlated with the diameter of the $\CP^1$ at the blowup. Therefore, as $\lambda\to 0$, we shrink the diameter of the blowup and end up with Euclidean space in the limit.

	\subsection{Smoothings of a $\log$-singularity with Ricci curvature bounded from below}\label{SectionLogSingRicciPos}
	
	For this family, consider instead of \eqref{OmegaLog} the metric
	\begin{equation}\label{OmegaLog2}
		(\omega^{\log}_\lambda)' = i\del\delbar (r^2 + \lambda\log (r^2 + 1)) = \begin{pmatrix}
			\frac{\lambda + 1 + r^2 + \lambda|z_2|^2 + r^4}{(r^2 + 1)^2} & - \frac{\lambda \overline{z_1}z_2}{(r^2 + 1)^2} \\
			- \frac{\lambda z_1\overline{z_2}}{(r^2 + 1)^2} & \frac{\lambda + 1 + r^2 + \lambda|z_1|^2 + r^4}{(r^2 + 1)^2}
		\end{pmatrix}.
	\end{equation}
	By the same computations as Section \ref{SectionLogSing} this metric has Ricci curvature bounded from below but otherwise has nonnegative scalar curvature, mass converging to 0, and converging to Euclidean space in the GH sense. Therefore, one can apply Theorems \ref{CorollaryStabNoRic} and \ref{CorollaryStabRicciBound} to this family. Unlike the other two, this family does not develop a singularity at $0\in \C^2$. 
	
	\subsection{Kähler Metrics Arising as Variations of The Euclidean Volume Form} 
	
	The next family arises as a solution to the complex Monge-Ampère equation with a spherical volume form that decays to the Euclidean volume form at infinity. Since the potential only depends on the radial direction $r$, the complex Monge-Ampère equation becomes an ordinary differential equation that can be explicitly solved. We will only write down one family of solutions with nonnegative scalar curvature, but many more can be found using this method.
	
	To start, we consider the volume form:
	\begin{equation*}
		f(r^2) \dvol_{\eucl}
	\end{equation*}
	where $f$ is a positive function that only depends on $r$. There is an explicit solution for the potential such that $i\del\delbar \phi$ is a Kähler metric with the given volume form. The potential is given by:
	\begin{equation}\label{PotentialSol}
		\phi(r) = \int_{r_0}^{r^2} \frac{1}{t} \left(\beta_0 + \int_{t_0}^s m t^{m-1} f(s) ds \right)dt
	\end{equation}
	where $r_0$, $t_0$, and $\beta_0$ are constants with $r_0$, $t_0 \geq 0$ and $\beta_0 \in \R$. For this discussion, we set $r_0 = t_0 = \beta_0 = 0$.\\
	
	Consider the function
	\begin{equation*}
		f_\lambda(r) \coloneqq \frac{r^2 + 10\lambda}{r^2 + \lambda}
	\end{equation*}
	to obtain a family of spherically symmetric volume forms decaying to the standard one at infinity. Write the solution of \eqref{PotentialSol} as $\phi_\lambda$, and denote the associated metric by
	\begin{align*}
		&\omega^{\vol}_\lambda = i\del\delbar \phi_\lambda =\small \frac{\sqrt 2}{\sqrt{g(r)}r^4(r^2+\lambda)} \times \\
		&\small \begin{pmatrix}
			|z_2|^2(r^2+ \lambda)g(r) + r^4|z_1|^2(\frac{1}{2}r^2+5\lambda) & \overline{z_1} z_2\left(-\left(r^2 + \lambda \right)g(r) + r^4 \left(\frac{1}{2}r^4 + 5\lambda \right) \right)\\
			z_1\overline{z_2}\left(-\left(r^2 + \lambda \right)g(r) + r^4 \left(\frac{1}{2}r^4 + 5\lambda \right) \right) & |z_1|^2(r^2+ \lambda)g(r) + r^4|z_2|^2(\frac{1}{2}r^2+5\lambda)
		\end{pmatrix},
	\end{align*}
	where
	\begin{equation*}
		g(r) = \frac{1}{2}r^4 + 9\lambda r^2 -9\lambda^2 \left(\log(r^2+\lambda)-\log(\lambda) \right)>0.
	\end{equation*}
	Again by considering the line ${z_2=0}$ and reading off the eigenvalues on the diagonals we conclude that all $\omega^{\vol}_\lambda$ are positive definite by $\text{U}(n)$-symmetry and hence Kähler metrics. Finally, $\omega_\lambda^{\vol} = \omega_{\eucl} + \OO(r^{-2})$ so the family is also AE. \\
	
	At $z_2=0$, or equivalently after diagonalizing, the Ricci curvature is
	\begin{equation*}
		\Ric_{\omega^{\vol}_\lambda} = \begin{pmatrix}
			\frac{90\lambda^3 - 9\lambda r^4}{(r^2+\lambda)^2(r^2+10\lambda)^2} & 0 \\ 0 & \frac{9\lambda}{(r^2+\lambda)(r^2+10\lambda)}
		\end{pmatrix}.
	\end{equation*}
	The second eigenvalue is nonnegative and unbounded from above, and the first one has infimum at
	\begin{equation*}
		r_{\min} = \sqrt{10^{\frac{2}{3}}\lambda + 10^{\frac{1}{3}}\lambda}.
	\end{equation*}
	Plugging this back into the first eigenvalue of $\Ric_{\omega^{\vol}_\lambda}$, the infimum becomes
	\begin{equation*}
		0>\frac{9(10-10^{\frac{1}{3}} - 10^{\frac{2}{3}})}{\lambda(10^{\frac{1}{3}}+ 10^{\frac{2}{3}} + 1)^2(10^{\frac{2}{3}} + 10^{\frac{1}{3}} + 10)^2} = \OO(\lambda^{-1})
	\end{equation*}
	as $\lambda \to 0$, so the Ricci curvature is not bounded from above or below.
	
	The scalar curvature is
	\begin{equation*}
		R_{\omega_\lambda^{\vol}} = \frac{9\lambda \left(r^8 + 20\lambda r^6 + 100\lambda^2 r^4 - 2g(r)r^4 + 20g(r)\lambda^2  \right)}{2r^2(r^2+\lambda)(r^2+10\lambda)^3g(r)} = \OO(r^{-6})
	\end{equation*}
	and bounded for every $\lambda>0$, hence $R_{\omega_\lambda^{\vol}}\in L^1_{\omega_\lambda^{\vol}}(\C^2)$. It is also positive by the estimate
	\begin{equation*}
		\frac{1}{2} r^4 \leq g(r) \leq \frac{1}{2}r^4 + 9\lambda r^2.
	\end{equation*}
	Using these estimates again and only keeping the highest and lowest order polynomial in the numerator and denominator by interpolation, we find that
	\begin{equation*}
		\m( \omega^{\vol}_\lambda) = \int_{\C^2} R_{\omega^{\vol}_\lambda} ({\omega^{\vol}_\lambda})^2 \leq C \int_0^\infty \frac{\lambda^2r^6 + \lambda^4r^2}{r^{12} + \lambda^4 r^4} r^3 dr =  \OO(\lambda)
	\end{equation*}
	for $\lambda\to 0$ and again using a computer for the computation. As the Ricci curvature is unbounded from below but the mass converges to 0, we can apply Theorem \ref{CorollaryStabNoRic} but not Theorem \ref{CorollaryStabRicciBound} to the family $(\omega_\lambda^{\vol})$. This allows us to conclude that the family GH converges to $(\mathbb{R}^4, d_{\text{eucl}})$ after removing suitable subsets along the family.
	\\
	
	As for $\omega_\lambda^{\log}$, we show directly that $(\C^2,\omega_\lambda^{\vol})$ GH converges to Euclidean space everywhere. As $\omega_\lambda^{\log}$ converges to $\omega_{\eucl}$ in $C^\infty_{\loc}(\C^2\setminus\{0\})$, it is enough to check this in a neighborhood of $0$. By spherical symmetry, take $p = (1,0)$ and the curve $\gamma\colon [0,1]\to \C^2, t\mapsto (t,0)$. Then for $z_2 = 0$:
	\begin{equation*}
		(\omega_\lambda^{\log})_{11} = \frac{|z_1|^2(|z_1|^2+10\lambda)}{\sqrt{2g(|z_1|)}(|z_1|^2 + \lambda)} \leq \frac{|z_1|^2(|z_1|^2+10\lambda)}{|z_1|^2(|z_1|^2 + \lambda)} = \frac{(|z_1|^2+10\lambda)}{(|z_1|^2 + \lambda)},
	\end{equation*}
	so
	\begin{align*}
		\length_{\omega^{\vol}_\lambda}(\gamma)&= \int_0^1|\gamma'(t)|_{\omega^{\vol}_\lambda} dt\\
		&= \int_0^1 \sqrt{(\omega^{\log}_\lambda)_{11}} dt\\ 
		&\leq \int_0^1  \sqrt{\frac{(|z_1|^2+10\lambda)}{(|z_1|^2 + \lambda}} dt \to 1, \quad \lambda \to 0.
	\end{align*}
	By replacing $g(|z_1|)$ by the upper bound $\frac{1}{2}|z_1|^4 + 9\lambda |z_1|^2$ above to obtain a lower bound converging to 1 as $\lambda\to 0$, it follows that
	\begin{equation*}
		\length_{\omega^{\vol}_\lambda}(\gamma) \to 1, \lambda\to 0,
	\end{equation*}
	and hence the sequence $(\omega^{\vol}_\lambda)$ converges in the GH sense to $(\R^4, d_{\eucl})$.

	\vspace{10mm}
	\printbibliography

\end{document}